\documentclass{amsart}
\usepackage{amssymb,amsmath,latexsym}
\numberwithin{equation}{section}
\newtheorem{theorem}{Theorem}[section]
\newtheorem{proposition}[theorem]{Proposition}
\newtheorem{corollary}[theorem]{Corollary}
\newtheorem{definition}[theorem]{Definition}

\newtheorem{lemma}[theorem]{Lemma}
\newtheorem{fact}[theorem]{Fact}
\newtheorem{example}[theorem]{Example}

\newtheorem{problem}[theorem]{Problem}

\newtheorem{observation}[theorem]{Observation}

\newcommand{\sims}{\stackrel{\ref{thm-strong}}{\sim}}

\newcommand{\simref}[1]{\stackrel{#1}{\sim}}
\newcommand{\simcomp}{\stackrel{d}{\sim}} 
\newcommand{\nceq}{\stackrel{nc}\sim}
\newcommand{\cceq}{\stackrel{cc}\sim}

\newcommand{\pil}{\pi_{\text{L}}} 
\newcommand{\pim}{\pi_\text{M}}
\newcommand{\pir}{\pi_{\text{R}}}
\newcommand{\pilp}{\pi'_{\text{L}}} 
\newcommand{\pimp}{\pi'_\text{M}}
\newcommand{\pirp}{\pi'_{\text{R}}}

\newcommand{\ttau}[2]{\tau_{#1}(#2)} 
\begin{document}

\pagenumbering{arabic}
\pagestyle{headings}
\def\sof{\hfill\rule{2mm}{2mm}}
\def\llim{\lim_{n\rightarrow\infty}}

\title{On multiple pattern avoiding set partitions}
\maketitle

\begin{center}
V\'\i t Jel\'\i nek\\
Computer Science Institute, Charles University in Prague,\\
Malostransk\'e n\'am\v est\'i 25, 118 00, Prague 1, Czechia

{\tt jelinek@iuuk.mff.cuni.cz}
\end{center}

\begin{center}
Toufik Mansour\\
Department of Mathematics, University of Haifa, 31905 Haifa, Israel

{\tt tmansour@univ.haifa.ac.il}
\end{center}

\begin{center}Mark Shattuck\\
Department of Mathematics, University of Haifa, 31905 Haifa, Israel

{\tt maarkons@excite.com}
\end{center}

\section*{Abstract}

We study classes of set partitions determined by the avoidance of multiple
patterns, applying a natural notion of partition containment that has been
introduced by Sagan. We say that two sets $S$ and $T$ of patterns are
equivalent if for each $n$ the number of partitions of size $n$ avoiding all the members of $S$
is the same as the number of those that avoid all the members of~$T$.

Our goal is to classify the equivalence classes among two-element pattern sets
of several general types. First, we focus on pairs of patterns
$\{\sigma,\tau\}$, where $\sigma$ is a pattern of size three with at least two
distinct symbols and $\tau$ is an arbitrary pattern of size~$k$ that
avoids~$\sigma$. We show that pattern-pairs of this type determine a small
number of equivalence classes; in particular, the classes have on average
exponential size in~$k$. We provide a (sub-exponential) upper bound for the
number of equivalence classes, and provide an explicit formula for the
generating function of all such avoidance classes, showing that in all cases
this generating function is rational.

Next, we study partitions avoiding a pair of patterns of the form $(1212,\tau)$,
where $\tau$ is an arbitrary pattern. Note that partitions avoiding $1212$
are exactly the non-crossing partitions. We provide several general equivalence
criteria for pattern pairs of this type, and show that these criteria account
for all the equivalences observed when $\tau$ has size at most six.

In the last part of the paper, we perform a full classification of the
equivalence classes of all the pairs $\{\sigma,\tau\}$, where $\sigma$ and
$\tau$ have size four.

\noindent{Keywords}: set partition, pattern avoidance, Wilf-equivalence class

\noindent{2010 Mathematics Subject Classification}: 05A18, 05A15, 05A19

\section{Introduction}

If $n \geq 1$, then a \emph{partition} of $[n]=\{1,2,\ldots,n\}$ is any
collection of nonempty, pairwise disjoint subsets, called \emph{blocks}, whose
union is $[n]$. (If $n=0$, then there is a single empty partition of
$[0]=\varnothing$ which has no blocks.) A partition $\Pi$ having exactly $k$
blocks is called a $k$-\emph{partition}. We will denote the set of all
$k$-partitions of $[n]$ by $P_{n,k}$ and the set of all partitions of $[n]$ by
$P_n$.  The number $n$ is referred to as \emph{the size} of a partition. A
partition $\Pi$ is said to be in \emph{standard form} if it is written
as $\Pi=B_1/B_2/\cdots$, where $\min(B_1)<\min(B_2)<\cdots$.  One may also
represent the partition $\Pi=B_1/B_2/\cdots/B_k \in P_{n,k}$, equivalently, by
the \emph{canonical sequential form} $\pi=\pi_1\pi_2\cdots\pi_n$, wherein $j\in
B_{\pi_j},\, 1\le j\le n$ (see, e.g., \cite{SW}). Throughout this paper, we will
represent set partitions by their canonical forms and consider the problem of
avoidance on these words.

For instance, the partition $\Pi=1,3,8/2,5/4,7/6 \in P_{8,4}$ has
the canonical sequential form $\pi=12132431$.  Note that
$\pi=\pi_1\pi_2\cdots\pi_n \in P_{n,k}$ is a {\em restricted growth function}
from $[n]$ to $[k]$ (see, e.g., \cite{Mi} for details), meaning that it
satisfies the following three properties: (i) $\pi_1=1$, (ii) $\pi$ is onto
$[k]$, and (iii) $\pi_{i+1}\leq\text{max}\{\pi_1,\pi_2,\ldots,\pi_i\}+1$ for all
$i$, $1\leq i \leq n-1$.  We remark that restricted growth functions are often
encountered in the study of set partitions \cite{SA,WW} as well as
other related topics, such as Davenport-Schinzel sequences \cite{DS,MS}.

 Let $\sigma=\sigma_1\sigma_2\cdots \sigma_n$ and
$\tau=\tau_1\tau_2\cdots\tau_m$ be two partitions, represented by their
canonical sequences.   We say that $\sigma$ \emph{contains} $\tau$ if $\sigma$
contains a subsequence that is order-isomorphic to $\tau$; that is, $\sigma$
has a subsequence $\sigma_{f(1)},\sigma_{f(2)},\ldots,\sigma_{f(m)}$, where
$1\leq f(1)<f(2)<\cdots<f(m)\leq n$, such that for each $i,j \in [m]$,  we have
$\sigma_{f(i)}<\sigma_{f(j)}$ if and only if $\tau_i<\tau_j$ and
$\sigma_{f(i)}>\sigma_{f(j)}$ if and only if $\tau_i>\tau_j$.  Otherwise, we say
that $\sigma$ \emph{avoids} $\tau$.  In this context, $\tau$ is usually called a
\emph{pattern}.  For example, the partition $\sigma$ avoids the pattern $1212$
if there exist no indices $i<j<k<\ell$ with
$\sigma_i=\sigma_k<\sigma_j=\sigma_{\ell}$ and avoids $1232$ if there exist no
such indices with $\sigma_i<\sigma_j=\sigma_\ell<\sigma_k$.

The concept of pattern-avoidance described above was introduced by Sagan
\cite{SA}, who considered, among other topics, the enumeration of partitions
avoiding patterns of size three. Several other notions of pattern-avoidance of
set partitions have been studied, see, e.g., the works of Klazar~\cite{K}, Chen
et al.~\cite{Ch}, or Goyt~\cite{G}.

We will use the following notation.  If $\{\tau_1,\tau_2,\ldots\}$ is a set of
patterns, then let $P_n(\tau_1,\tau_2,\ldots)$ and
$P_{n,k}(\tau_1,\tau_2,\ldots)$ denote the subsets of $P_n$ and $P_{n,k}$,
respectively, which avoid all of the patterns.  We will denote the cardinalities
of $P_n(\tau_1,\tau_2,\ldots)$ and $P_{n,k}(\tau_1,\tau_2,\ldots)$ by
$p_n(\tau_1,\tau_2,\ldots)$ and $p_{n,k}(\tau_1,\tau_2,\ldots)$, respectively.
From the definitions, note that $p_n(\tau_1,\tau_2,
\ldots)=\sum_{k\geq0}p_{n,k}(\tau_1,\tau_2,\ldots)$. In accordance with the
terminology first used for permutations (see, e.g., \cite{Ma}), we will say that
two sets of partition patterns $T=\{\tau_1,\tau_2,\ldots,\}$ and
$R=\{\rho_1,\rho_2,\ldots\}$ are (Wilf) \emph{equivalent}, denoted by $T\sim
R$, if $p_n(\tau_1,\tau_2,\ldots)=p_n(\rho_1,\rho_2,\ldots)$ for all $n \geq 0$.

The pattern avoidance question is a rather broad one in enumerative
combinatorics and has been the topic of much research, starting with Knuth
\cite{Kn} and Simion and Schmidt \cite{SI} on permutations.  See also, for
example, \cite{NZ,RWZ,MV}.  More recently, the problem has been considered on
further structures such as $k$-ary words and compositions.

Jel\'\i nek and Mansour ~\cite{parts} have determined all the equivalences
among singleton sets of patterns of size at most seven. In this paper, we focus
on classes of partitions determined by two forbidden patterns. We address three
main problems.  First, in Section~\ref{sec-3k} we
consider set partitions
avoiding a pair of patterns $\{\sigma,\tau\}$, where $\sigma$ is a pattern of
size three and $\tau$ is an arbitrary pattern not
containing~$\sigma$. The situation when $\sigma=111$ corresponds to
single-pattern
avoidance in partial matchings, which has been previously addressed~\cite{pp}.
We therefore restrict our attention to the cases when $\sigma\neq 111$.

We derive general criteria for Wilf-equivalence between pairs of patterns
$\{\sigma,\tau\}$ of this form. In particular, we show that when $\tau$ has size
$k$, these pairs form at most $\xi_k+1$ equivalence classes, where $\xi_k$
is the number of integer partitions having no summand equal to $2$
\cite[sequence A027336]{oeis}. This implies that on average the equivalence
classes have exponential size. For small values of $k$ (up to $k=
20$), we are able to verify that the estimate $\xi_k+1$
is sharp and all the equivalence classes may be described explicitly.

We also derive explicit formulas for the generating functions
$\sum_{n\geq0}p_n(\sigma,\tau)x^n$, where $\sigma\neq111$ is of size three
and $\tau$ is any pattern not containing $\sigma$. In particular, we show that
all these generating functions are rational.

Next, in Section~\ref{sec-nc}, we study the equivalences among pairs of
patterns of the form $\{1212,\tau\}$, where $\tau$ is a pattern that avoids
$1212$. Note that the partitions avoiding $1212$ correspond to the classical
non-crossing partitions. We may therefore regard this section as the
study of single-pattern avoidance among non-crossing partitions. We derive
several general criteria for equivalences of pairs of patterns of this form. It
turns out that some of the equivalence classes have size that is exponential
in the size of~$\tau$. We verify, with the help of computer enumeration, that
our criteria are sufficient to fully describe the equivalences among the pairs
$\{1212,\tau\}$ for $\tau$ of size at most six.

Finally, in Section~\ref{sec-44}, we perform a systematic classification of the
equivalences among the pairs $\{\sigma,\tau\}$, where $\sigma$ and $\tau$ are
distinct patterns of size four.  Partial results in this direction have already
been provided by previous research~\cite{MS3,MS2,MS1}. We provide
several new results concerning the avoidance of two or more patterns, including
ones involving infinite families of patterns.  By combining these results with
some specific cases which are worked out, we are able to provide a complete
solution to the problem of identifying all of the equivalence classes
corresponding to two patterns of size four.

We shall employ the following notation: if $\tau=\tau_1,\tau_2,\dotsc,\tau_n$ is
a sequence of numbers, then $\tau+1$ refers to the sequence
$\tau_1+1,\tau_2+1,\dotsc,\tau_n+1$. Also, if $a$ is a symbol and $q\ge 0$ an
integer, then $a^q$ refers to the constant sequence $a,a,\dotsc,a$ of
length~$q$.

\section{Avoiding a pattern of size three and another pattern}\label{sec-3k}

Our first main goal is to study classes of partitions that avoid a pair of
patterns $(\sigma,\tau)$, where $\sigma$ is a pattern of size three.

Note first that a set partition avoids $111$ if and only if each of its blocks
has size at most two. Such a partition is known as a \emph{partial matching}.
Pattern avoidance in partial matchings has already been addressed in a previous
paper~\cite{pp}. We therefore focus on the remaining patterns of size three,
that is, we assume $\sigma\in\{112,121,122,123\}$. We may also assume that
$\tau$ does not contain $\sigma$, otherwise $P_n(\sigma,\tau)=P_n(\sigma)$.

Let us remark that Sagan~\cite{SA} has shown that for any pattern $\sigma$ from
the set $\{112,121,122,123\}$, we have $p_n(\sigma)=2^{n-1}$.

Let us say that a pair of patterns $(\sigma,\tau)$ is a \emph{$(3,k)$-pair}
if $\sigma\in\{112,121,122,123\}$ and $\tau$ is a pattern of size $k$ that
avoids~$\sigma$. Our first results deal with general criteria for equivalences
among $(3,k)$-pairs. These criteria will apply to $(3,k)$-pairs for any value
of~$k$. For values of $k$ up to $k=20$,
we have verified that our criteria account for all equivalences among
the $(3,k)$-pairs. We conjecture that this is the case for larger~$k$ as well.

We also give an explicit formula for the generating function of partitions
avoiding an arbitrary given $(3,k)$-pair.

\subsection{The patterns $112$ and $121$}

Let us first consider the $(3,k)$-pairs $(\sigma,\tau)$, where $\sigma$ is
equal to either $121$ or~$112$. We will show that the two avoidance classes
$P_n(121)$ and $P_n(112)$ are closely related. More precisely, these two
classes form isomorphic posets under the containment relation.

Notice that a partition $\tau$ avoids $121$ if and only if $\tau$ is a weakly
increasing pattern of the form $1^{a_1} 2^{a_2}\cdots m^{a_m}$ for some $m\ge1$
and some sequence $a=(a_1,\dotsc,a_m)$ of positive integers. In particular,
there is a bijection between $121$-avoiding partitions of size $n$ and sequences
of positive integers whose sum is $n$.

Similarly, a partition $\tau$ avoids $112$ if and only if $\tau$ has the form
$12\dotsb m m^{a_m-1} (m-1)^{a_{m-1}-1}\dotsb 1^{a_1-1}$, for some $m\ge1$
and some sequence $a=(a_1,\dotsc,a_m)$ of positive integers. We use the term
\emph{composition} to refer to any finite sequence of positive integers. The
\emph{size} of a composition is the sum of its components, and the \emph{length}
of a composition is the number of its components.

For a composition $a=(a_1,\dotsc,a_m)$, let $\ttau{121}{a}$ denote the
$121$-avoiding pattern $1^{a_1} 2^{a_2}\cdots m^{a_m}$ and let $\ttau{112}{a}$
denote the 112-avoiding pattern $12\dotsb m m^{a_m-1} (m-1)^{a_{m-1}-1}\dotsb
1^{a_1-1}$. Note that $\ttau{112}{a}$ is the unique $112$-avoiding partition
with $m$ blocks whose $i$-th block has size~$a_i$, and similarly for
$\ttau{121}{a}$.

Let $a=(a_1,\dotsc,a_m)$ and $b=(b_1,\dotsc,b_k)$ be two compositions.
We say that $b$ \emph{dominates} $a$, if there is an $m$-tuple of indices
$i(1),i(2),\dotsc,i(m)$ such that $1\le i(1)<i(2)<\dotsb<i(m)\le k$, and
$a_j\le b_{i(j)}$ for each $j\in[m]$. In other words, $b$ dominates $a$ if $b$
contains a subsequence of length $m$ whose every component is greater than or
equal to the corresponding component of~$a$.

We present the following simple fact without proof.
\begin{observation}\label{obs-comp}
 For any two compositions $a$ and $b$, the following are equivalent:
\begin{itemize}
 \item $b$ dominates $a$,
\item $\ttau{112}{b}$ contains $\ttau{112}{a}$,
\item $\ttau{121}{b}$ contains $\ttau{121}{a}$.
\end{itemize}
\end{observation}

Observation~\ref{obs-comp} shows that the classes $P_n(112)$ and $P_n(121)$
ordered by containment and the set of all integer compositions ordered by
domination are three isomorphic posets, with size-preserving isomorphisms
identifying a composition $a$ with $\ttau{112}{a}$ and $\ttau{121}{a}$.

\begin{corollary}\label{cor-iso}
 For any integer composition $a$, the $(3,k)$-pairs $(112,\ttau{112}{a})$ and
$(121,\ttau{121}{a})$ are equivalent.
\end{corollary}

For two compositions $a$ and $a'$, let us write $a\simcomp a'$ if for every $n$,
the number of compositions of size $n$ dominating $a$ is equal to the number of
compositions of size $n$ dominating $a'$. Observation~\ref{obs-comp} implies
that $a\simcomp a'$ if and only if $(112,\ttau{112}{a})\sim
(112,\ttau{112}{a'})$ which is if and only if $(121,\ttau{121}{a})\sim
(121,\ttau{121}{a'})$.

For a composition $a=(a_1,\dotsc,a_m)$, let $M(a)$ denote the multiset
$\{a_1,\dotsc,a_m\}$.

\begin{lemma}\label{lem-multi1}
 Let $a$ and $a'$ be two compositions such that $M(a)=M(a')$. Then
$a\simcomp a'$.
\end{lemma}
\begin{proof}
It is enough to prove the lemma in the case when $a'$ is obtained from $a$ by
exchanging the order of two consecutive elements. Let $m$ be the length of $a$
and of~$a'$. Suppose $a$ equals $(a_1,\dotsc,a_m)$, and that $a'$ is obtained
from $a$ by exchanging the components $a_r$ and $a_{r+1}$ for some $r<m$, so
that we have
\begin{align*}
 a'=&(a_1,\dotsc,a_{r-1},a_{r+1},a_r,a_{r+2},\dotsc,a_m).
\end{align*}

We prove the lemma bijectively. Let $b=(b_1,\dotsc, b_k)$ be a
composition of size $n$ that dominates $a$. Let $i\in[k]$ be the smallest index
such that $(b_1,\dotsc,b_i)$ dominates $(a_1,\dotsc,a_{r-1})$. Let $j\in[k]$ be
the largest index such that $(b_{j+1},\dotsc,b_k)$ dominates
$(a_{r+2},\dotsc,a_m)$. Since $b$ dominates $a$, we know that $i+2\le j$ and
that $(b_{i+1},b_{i+2},\dotsc,b_j)$ dominates $(a_r,a_{r+1})$. Consider now a
composition $b'$ obtained from $b$ by reversing the order of the elements
$b_{i+1},b_{i+2},\dotsc,b_j$, that is,
\[
 b'=(b_1,\dotsc,b_i,b_j,b_{j-1},\dotsc,b_{i+1},b_{j+1},\dotsc,b_k).
\]
Clearly, $b'$ dominates $a'$, and the mapping $b\mapsto b'$ is a
size-preserving bijection between compositions that dominate $a$ and
those that dominate~$a'$.
\end{proof}

\begin{lemma}\label{lem-multi2}
Let $a=(a_1,\dotsc,a_m)$ be a composition, with $a_m=2$. Define another
composition $a'=(a_1,\dotsc,a_{m-1},1,1)$. Then $a\simcomp a'$.
\end{lemma}
\begin{proof}
Fix a size $n$. We provide a bijection $f$ between compositions of size $n$
that do not dominate $a$ and compositions of size $n$ that do not dominate
$a'$. Suppose that $b=(b_1,\dotsc,b_k)$ is a composition of size $n$ that does
not dominate $a$. If $b$ does not even dominate $(a_1,\dotsc,a_{m-1})$, then $b$
does not dominate $a'$, and we put $f(b)=b$.

Suppose now that $b$ dominates $(a_1,\dotsc,a_{m-1})$. Let $i\in[k]$ be the
smallest index such that $(b_1,\dotsc,b_i)$ dominates $(a_1,\dotsc,a_{m-1})$.
Since $b$ does not dominate $a$, we know that all of the components
$(b_{i+1},b_{i+2},\dotsc,b_k)$ must be equal to $1$. Define a new composition
$b'=f(b)$ obtained from $b$ by replacing all the components
$(b_{i+1},b_{i+2},\dotsc,b_k)$ with a single component equal to
$b_{i+1}+b_{i+2}+\dotsb+b_k$ (if $i=k$, then we put $b'=b$). Clearly, the new
composition $b'$ has size $n$ and does not dominate $a'$, and the mapping $f$
is the required bijection.
\end{proof}

From Lemmas~\ref{lem-multi1} and~\ref{lem-multi2}, we see that every
composition $a$ is $\simcomp$-equivalent to a composition $a'$ that has the
property that its components are weakly decreasing and none of them are equal
to~2. Let us call such a composition $a'$ a \emph{2-free integer partition}.
Let $\xi_k$ be the number of 2-free integer partitions of size $n$. Note that
the sequence $(\xi_k)_{k\ge 0}$ is listed as A027336 in the OEIS~\cite{oeis}.
Basic estimates on the number of integer partitions~(see, e.g., ~\cite{Andrews})
imply the bound $\xi_k=2^{O(\sqrt{k})}$.

Let $a=(a_1,a_2,\dotsc,a_\ell)$ be an integer composition, and let
$F_a(x,y)$ be the generating function for the number of partitions of
$[n]$ having exactly $k$ blocks and avoiding $\{112,\ttau{112}{a}\}$, i.e.,
$$F_a(x,y)=\sum_{n,k\geq0}p_{n,k}(112,\ttau{112}{a})x^ny^k.$$
We may give an explicit formula for $F_a(x,y)$ as follows.

\begin{theorem}\label{t1}
We have
\begin{equation}\label{e2}
F_a(x,y)=\sum_{j=0}^{\ell-1}\frac{x^{a_1+\cdots+a_j}y^j(1-x)}{\prod_{i=1}^{
j+1}\left(1-x(1+y)+x^{a_i}y\right)}.
\end{equation}
\end{theorem}
\begin{proof}
Let $b$ be the integer composition $(a_2,a_3,\dotsc,a_\ell)$.
Let $\pi$ be a nonempty partition from the set $P_{n,k}(112,\ttau{112}{a})$, and
let $r$ be the size of the first block of~$\pi$.
We consider the following two cases:
\begin{enumerate}
\item $1 \leq r \leq a_1-1$,
\item $r \geq a_1$.
\end{enumerate}
In the first case, $\pi$ must be of the form $1\pi'1^{r-1}$, where $\pi'$
is some partition on the letters $\{2,3,\ldots\}$ avoiding
$\{112,\ttau{112}{a}\}$, which
implies that the generating function counting these partitions is given by
$$xyF_a(x,y)+x^2yF_a(x,y)+\cdots+x^{a_1-1}yF_a(x,y)=\frac{x-x^{a_1}}{
1-x}yF_a(x,y).$$
In the second case, $\pi$ must be of the form $1\pi'1^{r-1}$, where $\pi'$
is now a partition on the letters $\{2,3,\ldots\}$ avoiding
$\{112,\ttau{112}{b}\}$ since
$r \geq a_1$.  Thus, the generating function counting the partitions in this
case is given by
$$x^{a_1}yF_b(x,y)+x^{a_1+1}yF_b(x,y)+\cdots=\frac{x^{a_1}y}{1-x}
F_b(x,y),$$
where we put $F_b(x,y)=0$ in the case $\ell=1$.  Adding the contributions from the two cases above gives
$$F_a(x,y)=1+\frac{x-x^{a_1}}{1-x}yF_a(x,y)+\frac{x^{a_1}y}{1-x}F_b(x,
y),$$
which may be rewritten as
\begin{equation}\label{e3}
F_a(x,y)=\frac{1-x}{1-x(1+y)+x^{a_1}y}+\frac{x^{a_1}y}{1-x(1+y)+x^{a_1}y}F_b(x,
y).
\end{equation}
Iterating recurrence \eqref{e3} yields \eqref{e2}, as desired.
\end{proof}

\subsection{The pattern $123$}

Consider now the $(3,k)$-pairs $(123,\tau)$, for a $123$-avoiding partition
$\tau$. Of course, a partition avoids $123$ if and only if it has at most two
blocks. We will distinguish two cases, depending on whether $\tau$ has a single
block or whether it has two blocks.

The first case is trivial:
\begin{observation}\label{obs-1231k}
A partition avoids the pair of patterns $(123,1^k)$ if and only it has at most
two blocks and each block has size at most $k-1$. In particular, the generating
function of the class $P_n(123,1^k)$ is given by the formula
\[
 \sum_{n\ge 0} p_n(123,1^k)x^n= 1+ \sum_{a=1}^{k-1} \sum_{b=0}^{k-1}
\binom{a+b-1}{b}x^{a+b}.
\]
\end{observation}

To deal with the pairs $(123,\tau)$, where $\tau$ has two blocks, we first prove
a more general theorem.

\begin{theorem}\label{thm-123}
Let $m\ge 2$ be an integer. Let $\tau=\tau_1\tau_2\dotsb\tau_k$ be a partition
with exactly $m$ blocks and with the property that $\tau_i=i$ for each
$i\in[m-1]$. Then the generating function $\sum_{n\ge 0}
p_n(12\dotsb(m+1),\tau)x^n$ is equal to
\[
  \left(\sum_{n\ge 0}
p_n(12\dotsb(m+1))x^n\right)
-\left(\frac{x}{1-(m-1)x}\right)^{k-m}\frac{x}{1-mx}
\prod_{j=1}^{m-1}\frac{x}{1-jx}.
\]
In particular, the generating function depends on the size of $\tau$ but not
on $\tau$ itself.
\end{theorem}
\begin{proof}
Let $Q_n$ be the set of partitions of size $n$ that avoid $12\dotsb(m+1)$
but contain $\tau$; in other words, $Q_n=P_n(12\dotsb(m+1))\setminus
P_n(12\dotsb(m+1),\tau)$. Define the generating function $G(x)=\sum_{n\ge 0}
|Q_n|x^n$. To prove the theorem, we need to prove the formula
\begin{equation}
G(x)=\left(\frac{x}{1-(m-1)x}\right)^{k-m}\frac{x}{1-mx}
\prod_{j=1}^{m-1}\frac{x}{1-jx}.\label{eq-123}
\end{equation}

Consider a partition $\pi\in Q_n$. Clearly $\pi$ must have exactly $m$ blocks.
Since $\pi$ contains $\tau$ as a pattern, and since $\tau$ has $m$ blocks as
well, we see that $\pi$ even contains $\tau$ as a subsequence. Recall that
$\tau_i=i$ for $i\in[m-1]$, that is, $\tau$ has the form $12\dotsb
(m-1)\tau_m\tau_{m+1}\dotsb\tau_k$. By fixing the leftmost occurrence of the
subsequence $\tau$ in $\pi$, we see that $\pi$ can be decomposed as
\[
 \pi = 1 w_1 2 w_2 3 w_3\dotsb (m-1) w_{m-1} \tau_m w_m
\tau_{m+1} w_{m+1} \dotsb w_{k-1} \tau_k w_k,
\]
where the $w_j$'s are determined as follows:
\begin{itemize}
 \item for $1\le j < m-1$, $w_j$ is an arbitrary word over the alphabet $[j]$,
\item for $m-1\le j < k$, $w_j$ is an arbitrary word over
$[m]\setminus\{\tau_{j+1}\}$, and
\item $w_k$ is an arbitrary word over $[m]$.
\end{itemize}
Conversely, any sequence with such a decomposition is an element of~$Q_n$. This
directly implies formula~\eqref{eq-123}.
\end{proof}

Applying Theorem~\ref{thm-123} to the case $m=2$, and noting that $\sum_{n\ge
0} p_n(123)x^n = (1-x)/(1-2x)$, we get the next result.

\begin{corollary}
For every $k$, the $(3,k)$-pairs of the form $(123,\tau)$ where $\tau$ has
two blocks are all equivalent, and the generating function of any such pair is
\[
\sum_{n\ge 0} p_n(123,\tau)x^n = \frac{1-x}{1-2x} -
\left(\frac{x}{1-x}\right)^{k-1}\frac{x}{1-2x}=\sum_{i=0}^{k-1}\left(\frac{x}{
1-x}\right)^i.
\]
\end{corollary}

Comparing the generating function of the previous corollary with the formula of
Theorem~\ref{t1}, we can say even more.

\begin{corollary}\label{cor-123-1}
For every $k$ and every partition $\tau\in P_k$ with two blocks, the
$(3,k)$-pair $(123,\tau)$ is equivalent to the $(3,k)$-pair $(112,12\cdots k)$.
\end{corollary}

\subsection{The pattern $122$}

Note that a partition $\tau$ avoids $122$ if and only if each block of $\tau$
except possibly the first one has size one, or equivalently, any number greater
than 1 appears at most once in~$\tau$.

We will show that for every $k$, all
the $(3,k)$-pairs of the form $(122,\tau)$ are equivalent to $(112,1^k)$. To
this end, we first describe a bijection between $122$-avoiding and
$123$-avoiding partitions which, under suitable assumptions, preserves
containment. Let $\tau=\tau_1\tau_2\dotsb\tau_k$ be a $122$-avoiding partition.
Define a new partition $f(\tau)=\tau'_1\tau'_2\dotsb\tau'_k$ by putting
$\tau'_i=1$ if $\tau_i=1$ and $\tau'_i=2$ if $\tau_i>1$. For example, if
$\pi=1123145$, then $f(\pi)=1122122$. Note that the mapping
$f$ defined by these properties is a bijection from the set of $122$-avoiding
partitions to the set of $123$-avoiding partitions.

\begin{lemma}\label{lem-122}
 Let $\tau=\tau_1\dotsb\tau_k$ be a $122$-avoiding partition with at least two
blocks, and let $\pi=\pi_1\dotsb\pi_n$ be any $122$-avoiding partition. Then
$\pi$ contains $\tau$ if and only if $f(\pi)$ contains~$f(\tau)$.
\end{lemma}
\begin{proof}
Let us write $f(\pi)=\pi'_1\dotsb\pi'_n$ and $f(\tau)=\tau'_1\dotsb\tau'_k$.

Assume that $\pi$ contains $\tau$, and let $i(1)<i(2)<\dotsb<i(k)$ be
indices such that the sequence $\pi_{i(1)}\pi_{i(2)}\dotsb\pi_{i(k)}$ is order
isomorphic to $\tau$. We may assume without loss of generality that $i(1)=1$. It
then follows that $\pi'_{i(1)}\pi'_{i(2)}\dotsb\pi'_{i(k)}$ is order-isomorphic
to $f(\tau)$ and hence $f(\pi)$ contains $f(\tau)$.

Conversely, assume that $f(\pi)$ contains $f(\tau)$. Since $\tau$ has at least
two blocks, $f(\tau)$ has exactly two blocks, and therefore $f(\pi)$
contains $f(\tau)$ even as a subsequence, not just as a pattern. This implies
that $\pi$ contains~$\tau$.
\end{proof}

\begin{proposition}\label{pro-122}
 For any partition $\tau\in P_k(122)$ with at least two blocks, the
$(3,k)$-pair $(122,\tau)$ is equivalent to the $(3,k)$-pair $(123,f(\tau))$.
\end{proposition}
\begin{proof}
 Lemma~\ref{lem-122} shows that $f$ maps $P_n(122,\tau)$
bijectively to $P_n(123,f(\tau))$.
\end{proof}

Combining Proposition~\ref{pro-122} with Corollary~\ref{cor-123-1}, we get the
following result.
\begin{corollary}\label{cor-122-A}
 For any $k$ and any partition $\tau\in P_k(122)$ with at least two blocks, the
$(3,k)$-pair $(122,\tau)$ is equivalent to the $(3,k)$-pair $(112,12\cdots k)$.
\end{corollary}

It remains to deal with $(3,k)$-pairs of the form $(122,1^k)$. It turns out
that these pairs are also equivalent to all the other $(3,k)$-pairs of the form
$(122,\tau)$.

\begin{proposition}\label{pro-122-1k}
The $(3,k)$-pairs $(122,1^k)$ and $(122,12\dotsb
k)$ are equivalent.
\end{proposition}
\begin{proof}
Note that a $122$-avoiding partition of $[n]$ is uniquely determined by
specifying which of the elements of the set $\{2,3,\dotsc,n\}$ belong to the
first block.

Thus, a $122$-avoiding partition avoids $1^k$ if and only if
its first block has at most $k-2$ elements from $\{2,3,\dotsc,n\}$, and it
avoids $12\dotsb k$ if and only if the complement of the first block has at most
$k-2$ elements from $\{2,3,\dotsc,n\}$. Clearly, in both cases there are exactly
$\sum_{i=0}^{k-2} \binom{n-1}{i}$ possibilities to specify the first block, and
therefore the whole partition.
\end{proof}

Combining Corollary~\ref{cor-122-A} and Proposition~\ref{pro-122-1k}, we obtain
the main result of this subsection.

\begin{corollary}\label{cor-122}
For any $k$, the $(3,k)$-pairs of the form $(122,\tau)$ are all equivalent, and
they are equivalent to the pair $(112,12\cdots k)$.
\end{corollary}

\subsection{Summary of equivalences among $(3,k)$-pairs}
\begin{table}
$$\begin{array}{|lll||c|}\hline
&&&\mbox{Generating function }\\
&(\rho,\tau)&&\mbox{for the sequence }\{p_n(\rho,\tau)\}_{n\geq0}\\\hline\hline
&&&\\[-9pt]
&(112,1231)\sim(112,1232)\sim(112,1233)&&\\
&\sim(121,1122)\sim(121,1123)\sim(121,1223)&&\\
&\sim(121,1233)\sim(121,1234)\sim(122,1111)&&\\
&\sim(122,1112)\sim(122,1121)\sim(122,1123)&&\\
&\sim(122,1211)\sim(122,1213)\sim(122,1231)&&\sum_{i=0}^3\frac{x^i}{(1-x)^i}\\
&\sim(122,1234)\sim(112,1221)\sim(123,1112)&&\\
&\sim(123,1121)\sim(123,1122)\sim(123,1211)&&\\
&\sim(123,1212)\sim(123,1221)\sim(123,1222)&&\\
&\sim(112,1234)&&\\\hline
&&&\\[-9pt]
&(123,1111)&&\sum_{i=1}^3\sum_{j=0}^3 \binom{i+j-1}{j} x^{i+j}\\\hline
&&&\\[-9pt]
&(121,1222)\sim(121,1112)&&\\
&\sim(112,1211)\sim(112,1222)&&\frac{1-x+x^3}{(1-x)(1-x-x^2)}\\\hline
&&&\\[-9pt]
&(121,1111)\sim(112,1111)&&\frac{1}{1-x-x^2-x^3}\\\hline
\end{array}$$
\caption{The equivalence classes of $(3,4)$-pairs}\label{tab-34}
\end{table}

Let us summarize the equi\-valences among $(3,k)$-pairs that follow from the
results established so far (see Table~\ref{tab-34} for an example with $k=4$).
\begin{itemize}
 \item There is an equivalence class containing all the patterns $(112,
\ttau{112}{a})$ and $(121,\ttau{121}{a})$ for all compositions $a$ of size $k$
all of whose components are equal to 1 or 2 (Observation~\ref{obs-comp},
Lemma~\ref{lem-multi1}, and Lemma~\ref{lem-multi2}). That same class also
contains all the $(3,k)$-pairs of the form $(122,\tau)$
(Corollary~\ref{cor-122}). By Corollary~\ref{cor-123-1}, the same class
also contains all the pairs of the form $(123,\tau)$, where $\tau$ is different
from $1^k$.
 \item The pair $(123,1^k)$ is not equivalent to any other $(3,k)$-pair.
There are only finitely many partitions avoiding both $123$ and $1^k$, whereas
any other $(3,k)$-pair is avoided by infinitely many partitions.
 \item For every 2-free integer partition $a$ of size $k$, there is an
equivalence class containing all the pairs from the set
\[\{(121,\ttau{121}{a'}),a'\simcomp a\}\cup\{(112,\ttau{112}{a'}), a'\simcomp
a\}.\]
Note that if $a=(1,1,\dotsc,1)$, then this class corresponds to the equivalence
class mentioned in the first item of this list. By Lemmas~\ref{lem-multi1}
and~\ref{lem-multi2}, any composition is $\simcomp$-equivalent to a 2-free
integer partition, therefore, the classes mentioned so far contain all the
$(3,k)$-pairs.
\end{itemize}

\begin{corollary}
 For each $k$ and each $(3,k)$-pair $(\sigma,\tau)$, the generating function of
$(\sigma,\tau)$-avoiders is rational, and can be computed explicitly.
\end{corollary}

\begin{corollary}
 For each $k\ge 3$, the $(3,k)$-pairs form at most $1+\xi_k$ equivalence
classes, where $\xi_k$ is the number of 2-free integer
partitions (A027336).
\end{corollary}

We do not know whether the bound of the previous corollary is tight or whether
there actually exist some more equivalences among the $(3,k)$-pairs. Note that
if such `hidden' equivalences exist, they must
involve $\tau$ of size at least 21, because for size 20 and less, we can check
(with the aid of a computer) that the
classes listed above are all non-equivalent. Also the additional equivalences
must involve $\sigma=112$ (or equivalently $\sigma=121$) because all the pairs
of the form $(122,\tau)$ are equivalent, and all the patterns of the form
$(123,\tau)$ are equivalent to them as well, except for $(123,1^k)$, which is
not equivalent to any other $(3,k)$-pair.

\begin{problem}
 Are there any more equivalences among the $(3,k)$-pairs of the form
$(112,\tau)$ other than those that we know about? Equivalently, are there any
two distinct 2-free integer partitions that are $\simcomp$-equivalent?
\end{problem}

\section{Pattern avoidance in non-crossing partitions}\label{sec-nc}

Our goal is to study partitions that avoid the pattern 1212 and another
pattern. Note that a partition avoids 1212 if and only if it is non-crossing.

We write $\sigma\nceq\tau$ if $(1212,\sigma)$ is equivalent to $(1212,\tau)$.
If $\sigma\nceq\tau$, we say that $\sigma$ and $\tau$ are \emph{nc-equivalent}
(`nc' stands for `non-crossing').

To simplify our notation, we will employ the following convention: whenever we
write $[\sigma]$ as a subsequence of a longer pattern $\pi$, we assume that
$[\sigma]$ refers to the sequence $\sigma+k$, where $k$ is the number of
distinct symbols of $\pi$ that appear before the first symbol of $\sigma$
in~$\pi$. Thus, for example, $1[112]1$ refers to the sequence $12231$, and
$\pi=11[121][112]$ should be understood as $\pi=11232445$.

Let us say that a set partition $\pi$ is \emph{connected}, if it cannot be
written as $\pi=\sigma[\tau]$ where $\sigma$ and $\tau$ are nonempty
partitions. Note that a non-crossing partition is connected if and only if its
last element belongs to the first block. For any set partition $\pi$, there is a
unique sequence of nonempty connected partitions $\sigma_1,\dotsc,\sigma_m$
such that $\pi=\sigma_1[\sigma_2][\sigma_3]\dotsb[\sigma_m]$. We call the
partitions $\sigma_i$ \emph{the components} of~$\pi$.

We say that two non-crossing partition patterns $\sigma$ and $\tau$ are
\emph{cc-equivalent}, denoted by $\sigma\cceq\tau$, if there is a bijection $f$
from the set of $(1212,\sigma)$-avoiding partitions to the set of
$(1212,\tau)$-avoiding partitions, such that for every non-crossing
$\sigma$-avoider $\pi$, the partition $f(\pi)$ has the same size and the same
number of components as~$\pi$. In particular, cc-equivalence is a refinement of
nc-equivalence.

 Suppose that
$\sigma=\sigma_1\dotsb\sigma_k$ and $\tau=\tau_1\dotsb\tau_n$ are two
partitions. We say that a sequence $I=(i(1),i(2),\dotsc,i(k))$ is an
\emph{occurrence} of $\sigma$ in~$\tau$, if $1\le i(1)<i(2)<\dotsb<i(k)\le n$
and $\tau_{i(1)}, \tau_{i(2)},\dotsc,\tau_{i(k)}$ is order-isomorphic
to~$\sigma$. We say that an occurrence $(i(1),\dotsc,i(k))$  of $\sigma$ in
$\tau$ is a \emph{leftmost occurrence} if
$i(k)$ has the smallest possible value among all occurrences of $\sigma$ in
$\tau$, or equivalently, $\sigma$ has no occurrence in
$\tau_1,\dotsc,\tau_{i(k)-1}$. We say that an occurrence $(i(1),\dotsc,i(k))$
of
$\sigma$ in $\tau$ is a \emph{topmost occurrence} if $\tau_{i(1)}$ has the
largest possible value among all the occurrences of $\sigma$ in $\tau$, or
equivalently, if the subsequence of $\tau$ formed by all the elements greater
than $\tau_{i(1)}$ is order-isomorphic to a $\sigma$-avoiding partition. If
$\sigma$ is the empty partition, we assume that the empty sequence is the unique
occurrence of $\sigma$ in $\tau$, and that this occurrence is both leftmost
and topmost.

For example, taking $\sigma=122$ and $\tau=11233245466233$, we see that $1,4,5$
and $2,4,5$ are the two leftmost occurrences of $\sigma$ in $\tau$, both
corresponding to a subsequence $133$ of $\tau$, while $8,10,11$ is the (in this
case unique) topmost occurrence of $\sigma$ in $\tau$, representing the
subsequence~$566$.

Let $I=(i(1),\dotsc,i(k))$ be a topmost occurrence of $\sigma$ in
$\tau$, with $b=\tau_{i(1)}$. Suppose that $i'$ is the smallest index such that
$\tau_{i'}=b$. Observe that replacing $i(1)$ with $i'$ yields another topmost
occurrence of $\sigma$ in~$\tau$.

\begin{lemma}\label{lem-left-top}
Let $\rho$, $\sigma$ and $\tau$ be non-crossing partitions of size
$k$, $\ell$ and $m$, respectively. Let $I=(i(1),\dotsc,i(k))$ be a leftmost
occurrence of $\rho$ in $\tau$, let $J=(j(1),\dotsc,j(\ell))$ be a topmost
occurrence of $\sigma$ in $\tau$, and let $H=(h(1),\dotsc, h(k+\ell))$ be any
occurrence of $\rho[\sigma]$ in~$\tau$. Then $(i(1),\dotsc,i(k),h(k+1),\dotsc,
h(k+\ell))$, as well as $(h(1),\dotsc,h(k),j(1),\dotsc,j(\ell))$ are both
occurrences of $\rho[\sigma]$ in~$\tau$. In particular, $\tau$ contains
$\rho[\sigma]$ if and only if $(i(1),\dotsc,i(k), j(1),\dotsc,j(\ell))$ is an
occurrence of $\rho[\sigma]$.
\end{lemma}
\begin{proof}
Let us assume that $\rho$, $\sigma$ and $\tau$ are nonempty, otherwise the
lemma is trivial.

Let us prove that $(i(1),\dotsc,i(k),h(k+1),\dotsc, h(k+\ell))$ is an occurrence
of $\rho[\sigma]$. Since we already know that $I$ is an
occurrence of $\rho$ and that $h(k+1),\dotsc,h(k+\ell)$ is an occurrence
of~$\sigma$, we only need to prove that $i(k)<h(k+1)$, and that every element
of $\tau_{i(1)},\tau_{i(2)}\dotsc,\tau_{i(k)}$ is smaller than any element of
$\tau_{h(k+1)},\tau_{h(k+2)},\dotsc,\tau_{h(k+\ell)}$.

Since $I$ is a leftmost occurrence of $\rho$, we know that
$i(k)\le h(k)$ and therefore $i(k)<h(k+1)$. We now show that
$\tau_{i(a)}<\tau_{h(k+b)}$ for every $a\in[k]$ and $b\in[\ell]$. Suppose that
we have $\tau_{i(a)}\ge \tau_{h(k+b)}$ for some $a\in[k]$ and $b\in[\ell]$. We
know that $\tau_{h(k)}<\tau_{h(k+b)}$. Let us write $x=\tau_{h(k)}$ and
$y=\tau_{h(k+b)}$. Let $i'$ and $j'$ be the indices of the first occurrences of
$x$ and $y$ in $\tau$, respectively. Since we know that $x<y\le\tau_{i(a)}$,
we know that $i'<j'\le i(a)<h(k)$. Thus, the four indices $i',j', h(k), h(k+b)$
are an occurrence of 1212 in $\tau$, contradicting the assumption that $\tau$
is non-crossing.

Let us now show that $h(1),\dotsc,h(k),j(1),\dotsc,j(\ell)$ is an occurrence of
$\rho[\sigma]$. Since $J$ is a topmost occurrence of
$\sigma$, we know that $\tau_{h(k+1)}\le \tau_{j(1)}$, and therefore
$\tau_{h(a)}<\tau_{j(b)}$ for any $a\in[k]$ and $b\in[\ell]$. To show that
$h(1),\dotsc,h(k),j(1),\dotsc,j(\ell)$ is an occurrence of
$\rho[\sigma]$, we thus only need to prove that $h(k)<j(1)$. Suppose that this
is not the case. Let us write $x=\tau_{h(k)}$ and $y=\tau_{h(k+1)}$, and let
$i'$ and $j'$ be the indices of first occurrences of $x$ and $y$ in $\tau$,
respectively. Since $x<y\le \tau_{j(1)}$, we know that $i'<j'\le j(1)$, showing
that $i',j', h(k), h(k+1)$ is an occurrence of 1212 in $\tau$, a contradiction.
\end{proof}

\begin{theorem}\label{thm-subst1}
 If $\sigma$ and $\tau$ are (possibly empty) non-crossing partitions, and if
$\rho$ and $\rho'$ are two cc-equivalent non-crossing partitions, then
$\sigma[\rho][\tau]\cceq\sigma[\rho'][\tau]$.
\end{theorem}
Note that in the previous theorem, cc-equivalence cannot be replaced by
nc-equivalence. For example, $11$ and $12$ are nc-equivalent partitions, but
$1[11]=122$ and $1[12]=123$ are not nc-equivalent.

\begin{proof}[Proof of Theorem~\ref{thm-subst1}]
Let us write $\alpha=\sigma[\rho][\tau]$ and $\alpha'=\sigma[\rho'][\tau]$.
We will define a bijection $f$ that maps $\alpha$-avoiding
non-crossing partitions of size $n$ to $\alpha'$-avoiding
non-crossing partitions of the same size, while preserving the number of
connected components.

 Let $\pi=\pi_1\dotsb\pi_n$ be a non-crossing partition on $n$ vertices. If
$\pi$ avoids $\sigma[\tau]$, then we may define $f(\pi)=\pi$. Suppose that
$\pi$ contains $\sigma[\tau]$. Let $I=(i(1),\dotsc,i(k))$ be the leftmost
occurrence of $\sigma$ in $\pi$, and $J=(j(1),\dotsc,j(\ell))$ the top-most
occurrence of $\tau$. Furthermore, assume that $j(1)$ is chosen as small as
possible, that is, $\tau_{j(1)}$ is the first element of its $\pi$-block.
By Lemma~\ref{lem-left-top}, we know that
$i(1),\dotsc,i(k),j(1),\dotsc,j(\ell)$ is an occurrence of~$\sigma[\tau]$.

Let us define $a=i(k)$ and $b=j(1)$. We will refer to the elements
$\pi_1,\pi_2,\dotsc,\pi_a$ as the \emph{left part} of $\pi$, while
$\pi_{a+1},\pi_{a+2},\dotsc,\pi_{b-1}$ are the \emph{middle part}, and
$\pi_b,\dotsc,\pi_n$ are the \emph{right part}. A block of $\pi$ is a
\emph{left block} (or \emph{middle block} or \emph{right block}) if its first
element appears in the left part of $\pi$ (middle part, right part,
respectively). We say that an element $\pi_i$ is an \emph{outlier} if it belongs
to a different part of $\pi$ than the first element of its block. In other
words, an outlier is an element of a left block belonging to the middle part or
right part, or an element of a middle block belonging to the right part.

Let $\pim$ denote the partition obtained from $\pi$ by deleting all the left
blocks and right blocks, and then by deleting all the outliers from the middle
blocks. In other words, $\pim$ consists of the elements of $\pi$ belonging to
middle blocks and to the middle part of~$\pi$. We will abuse the terminology by
identifying an element of $\pim$ with the corresponding element of~$\pi$.

It is clear that if $\pim$
contains $\rho$, then $\pi$ contains $\alpha$. We claim that the converse is
true as well, i.e., if $\pi$ contains $\alpha$ then $\pim$ contains $\rho$. To
see this, fix an occurrence $H$ of $\alpha$ in $\pi$, and write $H$ as a
concatenation $I'KJ'$, where $I'$, $K$ and $J'$ are occurrences of $\sigma$,
$\rho$ and $\tau$, respectively. By Lemma~\ref{lem-left-top}, $IKJ$ is also an
occurrence of~$\alpha$, which shows that all the indices in $K$ refer to the
middle part of~$\pi$. To see that for each $k\in K$, $\pi_k$ belongs to a
middle block, notice that $\pi_k>\pi_a$, and if $\pi_k$ belonged to a left
block, then the block containing $\pi_k$ would cross the block
containing~$\pi_a$. Thus, $K$ induces an occurrence of $\rho$ in $\pim$.

Suppose now that $\pi$ is an $\alpha$-avoiding partition, and therefore $\pim$
is a $\rho$-avoiding partition. Let $\kappa_1,\dotsc,\kappa_t$ be the connected
components of~$\pim$. Note that any outlier appearing in the
middle part of $\pi$ must appear in the `gap' between two components $\kappa_i$
and $\kappa_{i+1}$, otherwise we would have a crossing between a left block and
a middle block. In particular, each $\kappa_i$ corresponds to a consecutive
sequence of elements of~$\pi$. Note also that if the right part of $\pi$
contains an element from a middle block, then this middle block must correspond
to the first block of one of the components~$\kappa_i$.

Suppose that the cc-equivalence of $\rho$ and $\rho'$ is witnessed by a
bijection~$g$. Assume that the $\rho'$-avoiding partition $g(\pim)$ has
components $\kappa'_1,\dotsc,\kappa'_t$. We now define an $\alpha'$-avoiding
partition $\pi'$ having the same size and the same number of components
as~$\pi$. The left part of $\pi'$ is identical to the left part of $\pi$. In
the middle part, we replace the elements corresponding to $\kappa_i$ with the
elements corresponding to $\kappa'_i$, for each $i=1,\dotsc,t$. The elements
belonging to left blocks and appearing in the gap between $\kappa_i$ and
$\kappa_{i+1}$ will remain in the same block and will appear in the gap between
$\kappa'_i$ and $\kappa'_{i+1}$. Since we do not assume that each $\kappa_i'$
has the same size as $\kappa_i$, it may happen that the position of the gap
changes. We also do not assume that $g(\pim)$ has the same number of blocks as
$\pim$, so the numbering of right blocks may change as well. Finally, if in
$\pi$ the first block of $\kappa_i$ contains some elements in the right part of
$\pi$, then in $\pi'$ these elements will be inserted into the first block
of~$\kappa_i'$. We now define $f(\pi)=\pi'$. It is clear that $f$ has the
required properties.
\end{proof}

\begin{theorem}\label{thm-cycle}
 If $\sigma$ and $\rho$ are non-crossing partitions and $\sigma$ is connected,
then $\sigma[\rho]\cceq\rho[\sigma]$.
\end{theorem}
\begin{proof}
Let us define $\alpha=\sigma[\rho]$ and $\alpha'=\rho[\sigma]$. Of course, a
partition that avoids $\sigma$ must avoid both $\alpha$ and $\alpha'$. To
prove the theorem, we characterize the structure of a partition $\pi$
that contains $\sigma$ but not $\alpha$, as well as the structure of a partition
$\pi'$ containing $\sigma$ but not $\alpha'$. From the two characterizations, it
will be clear that the two classes are equinumerous and there is a bijection
between them preserving the number of components.

We will first describe the structure of an $\alpha$-avoiding non-crossing
partition~$\pi$ that contains $\sigma$. Let $I=(i(1),\dotsc,i(k))$ be a
leftmost occurrence of $\sigma$ in~$\pi$. Define $a=i(1)$ and $b=i(k)$. Note
that since $\sigma$ is connected, we know that $\pi_a$ and $\pi_b$ correspond to
the same block of~$\pi$. Choose $I$ in such a way that $a$ is the first element
of the block containing~$b$.

Define the left part of $\pi$ to be the elements strictly to the left of
$\pi_a$, the middle part to be the elements $\pi_a,\dotsc,\pi_b$, and the right
part to be the rest of~$\pi$. Define left blocks, middle blocks, right blocks
and outliers in the same way as in the previous proof. Note that there are no
outliers in the middle part of~$\pi$.

Let $\pil$ be the left part of $\pi$, let $\pim$ be the partition
order-isomorphic to the middle part of $\pi$, and let
$\pir$ be the partition formed by the elements in the right blocks of~$\pi$. By
the choice of $I$, we know that $\pil$ is $\sigma$-avoiding. It is not hard to
see that $\pi$ is $\sigma[\rho]$-avoiding if and only if $\pir$ is
$\rho$-avoiding. Let $\kappa_1,\dotsc,\kappa_t$ be the connected components of
$\pil$ ordered right-to-left, and let $\zeta_1,\dotsc,\zeta_u$ be the connected
components of~$\pir$ ordered left-to-right. Every outlier in $\pi$ is in the
right part of $\pi$ and its block is either the first block of $\pim$ or the
first block of one of the~$\kappa_i$. Each outlier must be placed between
the last vertex of $\zeta_i$ and the first vertex of $\zeta_{i+1}$ for some
$i\le u$, or between $\pi_b$ and the first vertex of~$\zeta_1$. Note that the
outliers form a weakly decreasing subsequence in~$\pi$.

Let $x_0$ be the number of outliers belonging to the first block of~$\pim$, and
for $i\in[t]$, let $x_i$ be the number of outliers from the first block of
$\kappa_i$. Thus, $\sum_{i=0}^t x_i$ is the number of all outliers in~$\pi$.
Let $y_0$ be the number of outliers appearing between the last element of
$\pim$ and the first element of $\zeta_1$, and for $i\in[u]$, let $y_i$ be the
number of outliers between the last element of $\zeta_i$ and the first element
of $\zeta_{i+1}$, with $y_u$ being the number of outliers to the right of
$\zeta_u$. The two sequences $x_0,\dotsc,x_t$ and $y_0,\dotsc,y_u$ determine
uniquely the position and value of the outliers in~$\pi$. Thus, $\pi$ is
uniquely determined by specifying $(\pil,\pim,\pir,
(x_i)_{i=0}^t,(y_i)_{i=0}^u)$. We may easily check that this gives a bijection
between $\sigma[\rho]$-avoiding partitions containing $\sigma$, and five-tuples
of the form $(\pil,\pim,\pir,(x_i)_{i=0}^t,(y_i)_{i=0}^u)$, where $\pil$ is a
$\sigma$-avoiding partition, $\pim$ is a connected partition that contains
$\sigma$ and every occurrence of $\sigma$ in $\pim$ intersects the last element
of $\pim$, $\pir$ is a $\rho$-avoiding partition, and $(x_i)_{i=0}^t$ and
$(y_i)_{i=0}^u$ are nonnegative integer sequences of the same sum, where $t$ is
the number of components in $\pil$ and $u$ is the number of components of~$\pir$.
Note that we rely on the fact that $\sigma$ is connected, which implies that if
$\pil$ avoids $\sigma$, then the leftmost occurrence of $\sigma$ in
$\pil[\pim]$ is contained in the component~$\pim$.

From the sequences $(x_i)_{i=0}^t$ and $(y_i)_{i=0}^u$, we may deduce
the number of components of~$\pi$ --- each component of $\pi$ is either equal to
$\kappa_i$ for some $i$, or equal to $\zeta_j$ for some $j$, or it contains
$\pim$. Moreover, $\kappa_i$ is a component of $\pi$ if and only if
$x_i=x_{i+1}=\dotsb =x_{t}=0$, and $\zeta_j$ is a component of $\pi$ if and only
if $y_j=y_{j+1}=\dotsb=y_u=0$.

Let us now provide an analogous analysis of the $\rho[\sigma]$-avoiding
partitions containing~$\sigma$. Let $\pi'$ be such a partition, and let
$I'=(i'(1),\dotsc,i'(k))$ be a topmost occurrence of $\sigma$ in $\pi'$, chosen
in such a way that $i'(1)$ is the first element of its block and $i'(k)$ is as
small as possible. Put $a=i'(1)$ and $b=i'(k)$, and define the left part, middle
part and right part of $\pi'$ in the same way as in the first part of the proof.
Let $\pilp$ be the left part of $\pi'$, let $\pimp$ be the middle part of
$\pi'$, and let $\pirp$ be the partition induced by the right blocks of~$\pi'$.
Then $\pilp$ is a $\rho$-avoiding partition and $\pirp$ is a $\sigma$-avoiding
partition. Suppose that $\kappa'_1,\dotsc,\kappa'_{t'}$ are the components of
$\pilp$ numbered right-to-left and $\zeta'_1,\dotsc,\zeta'_{u'}$ are the components
of~$\pirp$ numbered left-to-right. Let $x'_i$ be the number of outliers
belonging to the same $\pi$-block as the first vertex of $\kappa'_i$, with
$x'_0$ being the outliers belonging to the block of~$\pi_a$. Let $y'_i$ be the
number of outliers between $\zeta'_i$ and $\zeta'_{i+1}$, with $y'_0$ being the
number of outliers between $\pimp$ and~$\zeta_1$. Then $\pi'$ is uniquely
determined by $(\pilp,\pimp,\pirp,(x'_i)_{i=0}^{t'},(y'_i)_{i=0}^{u'})$, and the
$x'_i$s and $y'_i$s determine the number of components of~$\pi'$, in the same
way as in the case of~$\pi$. We see that by mapping $\pil$ to $\pirp$, $\pim$
to $\pimp$, $\pir$ to $\pilp$, $t$ to $u'$, $u$ to $t'$, $x_i$ to $y'_i$ and
$y_i$ to $x'_i$, we get the required bijection.
\end{proof}

\begin{theorem}\label{thm-permute}
 Let $\sigma_1,\dotsc,\sigma_k$ be a $k$-tuple of non-crossing partitions, and
let $p$ be a permutation of the set $\{1,2,\dotsc,k\}$. Then the partitions
$\sigma_1[\sigma_2]\dotsb[\sigma_k]$ and
$\sigma_{p(1)}[\sigma_{p(2)}]\dotsb[\sigma_{p(k)}]$ are cc-equivalent.
\end{theorem}
\begin{proof}
We may assume, without loss of generality, that all the
$\sigma_i$ are connected and that $p$ is a transposition of adjacent elements.
Suppose that for some $i<k$ we have $p(i)=i+1$, $p(i+1)=i$,
and $p(j)=j$ otherwise. By Theorem~\ref{thm-cycle}, we know that
$\sigma_i[\sigma_{i+1}]$ is cc-equivalent to~$\sigma_{i+1}[\sigma_i]$, and then
from Theorem~\ref{thm-subst1} we obtain the desired result, by putting
$\sigma=\sigma_1[\sigma_2]\dotsb[\sigma_{i-1}]$, $\rho=\sigma_i[\sigma_{i+1}]$,
$\rho'=\sigma_{i+1}[\sigma_i]$, and
$\tau=\sigma_{i+2}[\sigma_{i+3}]\dotsb[\sigma_k]$.
\end{proof}

\begin{theorem}\label{thm-conswap}
Let $\sigma_1,\dotsc,\sigma_k$ be a $k$-tuple of non-crossing partitions, let
$q<k$ be an index such that the partition $\sigma_q$ is empty, or connected,
or contains only singleton blocks. Then the partition
\[
\sigma=1[\sigma_1]1[\sigma_2]1\dotsb1[\sigma_{q-1}]1[\sigma_q]1[\sigma_{q+1}]
1[\sigma_{q+2}]1\dotsb1[\sigma_k]1
\]
is cc-equivalent to
\[
\sigma'=1[\sigma_1]1[\sigma_2]1\dotsb1[\sigma_{q-1}]1[\sigma_{q+1}]1[\sigma_q]
1[\sigma_{q+2}]1\dotsb1[\sigma_k]1.
\]
\end{theorem}
\newcommand{\lle}{{\le}}
\newcommand{\gge}{{\ge}}
\newcommand{\bpi}{\overline\pi}
\newcommand{\bbpi}{\widehat\pi}
\begin{proof}
We proceed by induction.  Fix an integer $n$, and suppose that for every $n'<n$
and for every $p$, the number of $\sigma$-avoiding partitions of size $n'$ with
$p$ components is equal to the number of $\sigma'$-avoiding such partitions. Let
$g$ be a bijection between $\sigma$-avoiders and $\sigma'$-avoiders of size less
than $n$, preserving size and number of components. We may assume, without loss
of generality, that $g$ has the property that $g(\pi)=\pi$ for any partition
$\pi$ that avoids both $\sigma$ and~$\sigma'$. We will define a bijection $f$
mapping $\sigma$-avoiders of size $n$ to $\sigma'$-avoiders of the same size and
number of components.

Let $\pi$ be a $\sigma$-avoiding partition of size~$n$. If $\pi$ is
disconnected, it can be written as $\pi=\pi_1[\pi_2]\dotsb[\pi_m]$ for $m>1$ and
$\pi_i$ connected. We then define $f(\pi)=g(\pi_1)[g(\pi_2)]\dotsb[g(\pi_m)]$.
This clearly satisfies all the claimed properties.

Assume now that $\pi$ is connected. Thus, $\pi$ can be uniquely
written as $\pi=1[\pi_1]1[\pi_2]1\dotsb1[\pi_m]1$ for some $\sigma$-avoiding
partitions~$\pi_i$.
We use the following terminology: for a partition $\rho$, an
occurrence $I=(i(1),\dotsc,i(\ell))$ of
$\rho$ in $\pi$ is a \emph{top-level occurrence} if it maps the elements of
the first block of $\rho$ to the elements of the first block of~$\pi$; in
other words, if $\pi_{i(1)}=1$. If $I$ is not a top-level occurrence, we say
that it is a \emph{deep occurrence}. Note that if $\rho$ is connected, then any
deep occurrence of $\rho$ in $\pi$ must correspond to an occurrence of $\rho$
in one of the partitions $\pi_1,\dotsc,\pi_m$.

For $i\le j\in[m+1]$, let $\pi(i,j)$ denote the partition
$1[\pi_i]1[\pi_{i+1}]1\dotsb1[\pi_{j-1}]1$, i.e., $\pi(i,j)$ is the
subpartition of $\pi$ between the $i$-th and $j$-th element of the first
block. For an integer $i\le m+1$, let $\pi(\lle i)$ denote the partition
$\pi(1,i)$ and $\pi(\gge i)$ be the partition $\pi(i,m+1)$. We apply analogous
notation for other connected partitions as well.

Let $\bpi_i$ denote the partition $g(\pi_i)$, and let $\bpi$ be the
partition $1[\bpi_1] 1[\bpi_2] 1\dotsb1[\bpi_m]1$. By induction, we know that
for any $i\in[m]$, $\bpi_i$ is $\sigma'$-avoiding and $\bpi_i$ contains $\sigma$
if and only if $\pi_i$ contains~$\sigma'$. Consequently, $\bpi$ has no deep
occurrence of $\sigma'$, and $\bpi$ has a deep occurrence of $\sigma$ if and
only if $\pi$ has a deep occurrence of~$\sigma'$. Using the fact that
$\bpi_i=\pi_i$ whenever
$\pi_i$ avoids both $\sigma$ and $\sigma'$, we also see that for any $j\in[k]$,
$\bpi_i$ contains $\sigma_j$ if and only if $\pi_i$ does, and more generally,
for any $h,i\in[m+1]$, $\pi(h,i)$ has a top-level occurrence of $1[\sigma_j]1$
if and only if $\bpi(h,i)$ does. Consequently, $\bpi$ has no top-level
occurrence of $\sigma$, and $\bpi$ has a top-level occurrence of $\sigma'$ if
and only if $\pi$ does.

Let $a\in[m+1]$ be the smallest index such that $\bpi(\lle a)$ has a top-level
occurrence of $\sigma(\lle q)$, and let $b\in[m+1]$ be the largest index such
that $\bpi(\gge b)$ has a top-level occurrence of $\sigma(\gge q+2)$. If such
$a$ or $b$ do not exist, or if $a+2> b$, then $\bpi$ has no top-level
occurrence of either $\sigma$ or $\sigma'$, and we define~$f(\pi)=\bpi$.

Suppose that $a+2\le b$, and let $c$ be the smallest integer from
$\{a+1,a+2,\dotsc,b\}$ such that
$\bpi(a,c)$ has a top-level occurrence
of~$1[\sigma_q]1$. If no such $c$ exists, we again put~$f(\pi)=\bpi$.
Otherwise, define a partition $\bbpi=1[\bbpi_1] 1 [\bbpi_2]1\dotsb1[\bbpi_m]1$, by
putting
$\bbpi(\lle a)=\bpi(\lle a)$, $\bbpi(\gge b)=\bpi(\gge b)$,
$\bbpi(a,a+b-c)=\bpi(c,b)$, and $\bbpi(a+b-c,b)$ being equal to
$1[\bpi_{c-1}]1[\bpi_{c-2}]1\dotsb1[\bpi_{a+1}]1[\bpi_a]1$. Notice that
$\bbpi(a+b-c,b)$ has a top-level occurrence of $1[\sigma_q]1$, while
$\bbpi(a+b-c+1,b)$ does not (here we use the assumption that $\sigma_q$ is
empty, or connected, or only contains singleton blocks). We also know that
$\bpi(c,b)$ has no top-level occurrence of $1[\sigma_{q+1}]1$, because $\bpi$ has
no top-level occurrence of~$\sigma$. This implies that $\bbpi$ has no top-level
occurrence of $\sigma'$, and therefore $\bbpi$ is a $\sigma'$-avoiding
partition. We then define $f(\pi)=\bbpi$. It is easy to check that $f$ is a
bijection between $\sigma$-avoiding and $\sigma'$-avoiding partitions of size
$n$ which preserves the number of components.
\end{proof}

\begin{theorem}\label{thm-subst2}
Let $\sigma_1,\dotsc,\sigma_k$ be a $k$-tuple of non-crossing partitions, let
$q<k$ be an index, and let $\sigma'_q$ be a partition cc-equivalent to
$\sigma_q$. Then the pattern
\[
\sigma=1[\sigma_1]1[\sigma_2]1\dotsb1[\sigma_{q-1}]1[\sigma_q]1[\sigma_{q+1}]
1\dotsb1[\sigma_k]1
\]
is cc-equivalent to
\[
\sigma'=1[\sigma_1]1[\sigma_2]1\dotsb1[\sigma_{q-1}]1[\sigma'_q]1[\sigma_{q+1}]
1\dotsb1[\sigma_k]1.
\]
\end{theorem}
\begin{proof}
As in the proof of Theorem~\ref{thm-conswap}, we proceed by induction. Suppose
again that $n$ is given, and that there is a bijection $g$ mapping the
$\sigma$-avoiders of size less than $n$ to $\sigma'$-avoiders of the same size
and same number of components. Suppose also that $g(\pi)=\pi$ for any partition
that avoids both $\sigma$ and~$\sigma'$. We define a bijection $f$ mapping
$\sigma$-avoiders of size $n$ to $\sigma'$-avoiders of the same size and number
of components. Let $h$ be a mapping from $\sigma_q$-avoiding partitions to
$\sigma'_q$-avoiding partitions which witnesses the cc-equivalence of
$\sigma_q$ and $\sigma'_q$.

Let $\pi$ be a $\sigma$-avoiding partition of size~$n$. If $\pi$ is disconnected
with components $\pi_1,\dotsc,\pi_m$, we define $f(\pi)$ to be the partition
with components $g(\pi_1),\dotsc, g(\pi_m)$. Suppose now that $\pi$ is
connected, and has the form $1[\pi_1]1\dotsb1[\pi_m]1$.

We will define a new partition $\bpi$ that has no top-level occurrence of
$\sigma'$. Let $a$ be the smallest integer such that $\pi(\lle a)$ has a
top-level occurrence of $\sigma(\lle q)$, and let $b$ be the largest integer
such that $\pi(\gge b)$ has a top-level occurrence of $\sigma(\gge q+1)$. If
such $a$ or $b$ does not exist, or if $a\ge b$, we define~$\bpi=\pi$. Otherwise,
let $\rho$ be the partition $\pi_{a}[\pi_{a+1}]\dotsb[\pi_{b-1}]$, and let $p_i$
be the number of connected components of $\pi_i$, so that $\rho$ has
$p_a+p_{a+1}+\dotsb+p_{b-1}$ components. Note that $\rho$ avoids
$\sigma_q$. Define $\rho'=h(\rho)$, and write $\rho'$ as
$\rho'=\bpi_a[\bpi_{a+1}]\dotsb[\bpi_{b-1}]$, where each $\bpi_i$ is chosen so that it has exactly $p_i$ components. We now define the partition
$\bpi=1[\bpi_1]1\dotsb1[\bpi_m]1$ by putting $\bpi(\lle a)=\pi(\lle a)$,
$\bpi(\gge b)=\pi(\gge b)$, and $\bpi(a,b)$ is determined by the partitions
$\bpi_i$ obtained from~$\rho'$.

Note that $\bpi$ has no top-level occurrence of $\sigma'$. Also, $\bpi(a,b)$
has no deep occurrence of $\sigma'$, because it does not even have a deep
occurrence of~$\sigma'_q$. Define now a partition $\bbpi=1[\bbpi_1]
1\dotsb1[\bbpi_m]1$ by putting $\bbpi_i=\bpi_i$ for each $a\le i<b$, and
$\bbpi_i=g(\bpi_i)$ for each $i<a$ and $i\ge b$. Then $\bbpi$ has no deep
occurrence of $\sigma'$. Using the fact that $g(\bpi_i)=\bpi_i$ whenever
$\bpi_i$ avoids both $\sigma$ and $\sigma'$, we can also see that $a$ is the
smallest index such that $\bbpi(\lle a)$ has a top-level occurrence of
$\sigma'(\lle q)$, and $b$ is the largest index such that $\bbpi(\gge b)$ has a
top-level occurrence of $\sigma'(\gge q+1)$. We put $f(\pi)=\bbpi$, and easily
see that $f$ is the required bijection.
\end{proof}

In the rest of this section, we will often employ generating functions as tools
in our proofs. Let us therefore fix the following notation. For a partition
$\pi$, we let $NC(x;\pi)$ denote the generating function of the set of
non-crossing $\pi$-avoiding partitions, and we let $C(x;\pi)$ denote the
generating function of the set of nonempty connected non-crossing
$\pi$-avoiding partitions.

\begin{theorem}\label{thm-1t1s}
Let $\sigma$ and $\tau$ be two possibly empty connected partitions. Then
$1[\sigma]1[\tau]$ and $1[\tau]1[\sigma]$ are nc-equivalent.
\end{theorem}

In the previous theorem, the assumption that $\sigma$ and $\tau$ are connected
is necessary, as shown, e.g., by the two patterns $1[1]1[12]=12134$ and
$1[12]1[1]=12314$, which are not nc-equivalent. Also, nc-equivalence in the
conclusion cannot in general be replaced with cc-equivalence. For example,
taking $\sigma$ empty and $\tau=1$, we see that $1[\sigma]1[\tau]=112$, while
$1[\tau]1[\sigma]=121$. Since $112$ and $121$ do not have the same number of
components, it is easy to see that they cannot be cc-equivalent.

\begin{proof}[Proof of Theorem~\ref{thm-1t1s}]
Let us first deal with the situation where both $\sigma$ and $\tau$ are
nonempty. Let $G(x; \sigma,\tau)$ denote the generating function of non-crossing
partitions that avoid $1[\sigma]1[\tau]$ but contain $\sigma[\tau]$, in other
words,
\[
G(x;\sigma,\tau)=NC(x;1[\sigma]1[\tau])-NC(x;\sigma[\tau]).
\]

We know from Theorem~\ref{thm-cycle} that 
$NC(x,\sigma[\tau])=NC(x,\tau[\sigma])$. Therefore, to show
that $1[\sigma]1[\tau]$ is nc-equivalent to $1[\tau]1[\sigma]$, it is enough to
prove that $G(x;\sigma,\tau)=G(x;\tau,\sigma)$. We will derive a formula for
$G(x;\sigma,\tau)$ from which the previous identity will easily follow.

\newcommand{\cC}{\mathcal{C}}
\newcommand{\cM}{\mathcal{M}}
Note that if $\rho$ is a connected partition, a non-crossing partition $\pi$
avoids $\rho$ if and only if each component of $\pi$ avoids~$\rho$. In
particular, we have the identity
$NC(x;\rho)=1/(1-C(x;\rho))$. We say that a partition $\pi=\pi_1,\dotsc,\pi_n$
is \emph{$\rho$-minimal} if it is connected, non-crossing, contains $\rho$, but
avoids $1[\rho]1$. Let $M(x;\rho)$ be the generating function of the set of
$\rho$-minimal partitions.

Suppose that $\pi=\pi_1,\dotsc,\pi_n$ is a non-crossing partition that avoids
$1[\sigma]1[\tau]$ and contains~$\sigma[\tau]$. Let $I=(i(1),\dotsc,i(k))$ be a
leftmost occurrence of $\sigma$ in $\pi$, chosen in such a way that $i(1)$ is
as small as possible. This implies that $\pi_{i(1)}$ is the leftmost element of
its $\pi$-block. Let us write $a=i(1)$.

Let $J=(j(1),...,j(\ell))$ be a topmost occurrence of $\tau$. Choose $J$ in
such a way that $j(1)$ is as small as possible, and write
$b=j(1)$. Then $\pi_b$ is the leftmost element of its
block. Let $\pi_c$ be the rightmost element of the block containing $\pi_b$.
Then $\pi_b,\pi_{b+1},\dotsc,\pi_c$ is order-isomorphic to
a $\tau$-minimal partition, because if it contained a copy of $1[\tau]1$, it
would contradict the topmost choice of~$J$.

Let $\pil$ denote the partition $\pi_1,\pi_2,\dotsc,\pi_{b-1}$. Note that
$\pil$ avoids $1[\sigma]1$. Let $\pil^1,\pil^2,\dotsc,\pil^m$ be the connected
components of~$\pil$. Let $\pil^q$ be the component of $\pil$ containing the
vertex $\pi_a$. Note that $\pi_a$ must be the leftmost vertex of $\pil^q$,
otherwise $\pil^q$ would contain $1[\sigma]1$. We see that $\pil^q$ is a
$\sigma$-minimal partition. Note also that all the components preceding $\pil^q$
must avoid $\sigma$, since $\pil^q$ contains the leftmost occurrence
of~$\sigma$.

We say that an element $\pi_i$ of $\pi$ is an \emph{outlier}, if $i>c$ and
$\pi_i<\pi_c$. In other words, an outlier is an element that does not belong to
$\pil$, but belongs to a $\pi$-block whose leftmost element belongs to~$\pil$.
Note that if $\pi_i$ is an outlier, then the $\pi$-block containing $\pi_i$
intersects a unique component $\pil^j$ of $\pil$, and it is the first block of
of~$\pil^j$; we then say that $\pi_i$ is an \emph{outlier from~$\pil^j$}.

For a component $\pil^j$ of $\pil$, define \emph{the zone} of $\pil^j$, denoted
by $Z_j$, inductively as follows. If $\pil^j$ has no outlier, then $Z_j$ is
empty, otherwise $Z_j$ is the sequence $\pi_g,\pi_{g+1},\dotsc,\pi_h$, where
$\pi_g$ is the leftmost outlier of $\pil^j$ and $\pi_h$ is the rightmost
vertex not belonging to $Z_1\cup Z_2\cup \dotsb\cup Z_{j-1}$. Let
$\pi_d$ be the rightmost vertex of $\pi$ not belonging to any zone. The
zones $Z_1,\dotsc,Z_m$ form a disjoint collection of subsequences whose union
is~$\pi_{d+1},\dotsc,\pi_n$. Each zone $Z_j$ is order-isomorphic to a partition
of the form $1[\rho_1] 1[\rho_2] 1\dotsb 1[\rho_r]$ in which each occurrence of
1 corresponds to an outlier from $\pil^j$, and each $\rho_i$ is a
$\tau$-avoiding partition formed by non-outliers. The generating function of
such partitions is
\[
 Z(x)= \frac{1}{1-xNC(x;\tau)}= \frac{1}{1-\frac{x}{1-C(x;\tau)}}.
\]
Note also that the elements
$\pi_{c+1},\pi_{c+2},\dotsc,\pi_d$ (which do not belong to any zone and do not
contain any outliers) are order-isomorphic to a $\tau$-avoiding partition.

We claim that the generating function of all the non-crossing partitions $\pi$
avoiding $1[\sigma]1[\tau]$, containing $\sigma[\tau]$, having $m$ components in
$\pil$, and with the component $\pil^q$ containing the leftmost occurrence of
$\sigma$ is equal to
\[
 \left(C(x;\sigma) Z(x)\right)^{q-1}
M(x;\sigma)Z(x)\left(C(x;1[\sigma]1)Z(x)\right)^{m-q}M(x;\tau)\frac{1}{
1-C(x;\tau)}.
\]
To see this, note first that each factor $C(x;\sigma)Z(x)$ corresponds to one
of the first $q-1$ components of $\pil$, together with its zone. Next, the
factor $M(x;\sigma)Z(x)$ corresponds to the possible choices for the component
$\pil^q$ and its zone. The factor $C(x;1[\sigma]1)Z(x)$ corresponds to a
component $\pil^i$ for $i>q$, together with its zone. The factor $M(x;\tau)$
corresponds to the elements from $\pi_b$ to $\pi_c$, and the next factor
$(1-C(x;\tau))^{-1}$ corresponds to the elements $\pi_{c+1},\dotsc,\pi_d$.

Summing the above expression for all possible $m\ge 1$ and $q\in[m]$, we obtain
\begin{equation}
 G(x;\sigma,\tau)=\frac{1}{1-C(x;\sigma)Z(x)}M(x;\sigma)Z(x)\frac{1}{1-C(x;1[
\sigma]1)Z(x)}M(x;\tau)\frac{1}{1-C(x;\tau)}.\label{eq-G}
\end{equation}
Using the identity
\[
C(x;1[\rho]1)=x+\frac{x^2}{1-x-C(x;\rho)}=\frac{x}{1-\frac{x}{1-C(x;\rho)}},
\]
which is valid for any connected non-crossing partition~$\rho$, we define two
auxiliary expressions, both of which are symmetric in $\sigma$ and $\tau$:
\begin{align*}
 F_1(x;\sigma,\tau)&=\frac{1}{1-C(x;1[\sigma]1)Z(x)}
=\frac{1}{1-\frac{
x(1-C(x;\sigma))(1-C(x;\tau))} { (1-x-C(x;\sigma))(1-x-C(x;\tau))}}\\
\intertext{and}
F_2(x;\sigma,\tau)&=\frac{Z(x)}{(1-C(x;\sigma)Z(x))(1-C(x;\tau))}\\
&=\frac{1}{1-x-C(x;\sigma)-C(x;\tau)+C(x;\sigma)C(x;\tau)}.
\end{align*}

With this notation, \eqref{eq-G} simplifies into
\[
 G(x;\sigma,\tau)=M(x;\sigma)M(x;\tau)F_1(x;\sigma,\tau)F_2(x;\sigma,\tau),
\]
This makes it clear that $G(x;\sigma,\tau)=G(x;\tau,\sigma)$, completing the
proof for the case when both $\sigma$ and $\tau$ are nonempty.

It remains to deal with the case when $\sigma$ or $\tau$ is empty, i.e., to
show that $1[\tau]1\nceq 11[\tau]$ for any connected~$\tau$. The generating
function of $1[\tau]1$-avoiding non-crossing partitions is equal to
\[
\frac{1}{1-C(x;1[\tau]1)}=\frac{1-x-C(x;\tau)}{1-2x-C(x;\tau)+xC(x;\tau)
}.
\]
Let us now sketch the argument for the pattern $11[\tau]$. Partitions avoiding
$11$ have generating function $1/(1-x)$. Let $\pi$ be a partition containing
$11$ and avoiding $11[\tau]$. Let $I=(a,b)$ be the leftmost occurrence of $11$
in~$\pi$. All the elements $\pi_1,\pi_2,\dotsc,\pi_{b-1}$ belong to distinct
blocks of~$\pi$. An element $\pi_i$ is an \emph{outlier} if $i>b$ and $\pi_i\le
\pi_a$. The elements $\pi_{b+1},\dotsc,\pi_n$ that are not outliers form a
$\tau$-avoiding partition. We may define zones $Z_1,\dotsc,Z_a$ in analogy to
the previous case. All elements of $\pi$ to the right of the leftmost
outlier (inclusive) belong to a unique zone. This yields a generating function
\[
 \frac{1}{1-xZ(x)}\frac{x^2 Z(x)}{1-x}\frac{1}{1-C(x;\tau)},
\]
where the factor $1/(1-x)$ counts the elements between $\pi_a$ and $\pi_b$,
while the factor $(1-C(x;\tau))^{-1}$ counts subpartitions formed by the
elements to the right of $\pi_b$ but to the left of the leftmost zone.

Adding $1/(1-x)$ to the above expression and simplifying shows that
$11[\tau]$-avoiding partitions have the same generating function as
$1[\tau]1$-avoiding partitions.
\end{proof}

\begin{theorem}\label{thm-add1}
 Let $\sigma$ be a connected partition, and let $\tau$ be a partition of the
form $1[\rho]$ for some partition $\rho$. If $\sigma$ and $\tau$ are
nc-equivalent, then $1[\sigma]$ and $\tau1=1[\rho]1$ are nc-equivalent as well.
\end{theorem}

\begin{proof}
Let $NC(x;\pi)$ denote the generating function of the set of $\pi$-avoiding
non-crossing partitions. By assumption, we have $NC(x;\sigma)=NC(x;\tau)$.

Note that a partition $\pi$ avoids $1[\sigma]$ if and only if it can be
written as 
\[\pi=1[\pi^1]1[\pi^2]1\dotsb1[\pi^k]\]
for some $k$, where each $\pi^i$
is a $\sigma$-avoiding partition (here we use the fact that $\sigma$ is
connected). Therefore, we have the identity
\[
 NC(x;1[\sigma])=\frac{1}{1-xNC(x;\sigma)}.
\]

Consider now the partition $\tau1=1[\rho]1$. Since this partition is connected,
we see that $\pi$ avoids $\tau1$ if and only if each component of $\pi$ avoids
$\tau1$. Moreover, a connected partition $\pi=\pi_1,\dotsc,\pi_n$ avoids
$\tau1$ if and only if $\pi_1,\dotsc,\pi_{n-1}$ avoids $\tau$ (here we use the
fact that $\tau$ has only one occurrence of the symbol $1$, and that $\pi_n=1$
because $\pi$ is connected). This implies the identity
\[
 NC(x;\tau1)=\frac{1}{1-xNC(x;\tau)}.
\]

Since $NC(x;\sigma)=NC(x;\tau)$, we get that $NC(x;1[\sigma])=NC(x;\tau1)$.
\end{proof}

As an example of an application of Theorem~\ref{thm-add1}, consider the
partitions $\sigma=11$ and $\tau=12$. Since these two partitions are
nc-equivalent, the theorem implies
that $1[\sigma]=122$ and $\tau1=121$ are nc-equivalent as well. We may in fact
apply the theorem again to this new pair of patterns, and obtain that
$1221\nceq 1232$, and a third application reveals that $12332\nceq12321$.
Generalizing this example into a straightforward induction argument, we get the
next corollary.

\begin{corollary}\label{cor-add1}
For any $k$, the pattern $12\dotsb(k-1) k k (k-1) \dotsb 3 2$ is nc-equivalent
to the pattern $1 2 \dotsb (k-1) k (k-1) \dotsb 2 1$, and the pattern $1 2
\dotsb (k-1) k k (k-1)\dotsb  2 1$ is nc-equivalent to $ 1 2 \dotsb k (k+1) k
\dotsb 3 2$.
\end{corollary}

Note that the partitions avoiding $1 2 \dotsb (k-1) k (k-1) \dotsb
2 1$ are precisely those that do not have a $k$-tuple of pairwise nested blocks.

\begin{theorem}\label{thm-3nest}
The partitions $12333$ and $12321$ are nc-equivalent. In other words, the
non-crossing partitions whose every block has size at most two except possibly
the first two blocks are equinumerous with the non-crossing partitions that
have no 3-nesting.
\end{theorem}

\begin{proof}
 We again let $NC(x;\pi)$ denote the generating function of non-crossing
$\pi$-avoiding partitions, and let $C(x;\pi)$ be the generating function of
nonempty connected non-crossing $\pi$-avoiding partitions. As we have already pointed
out before, for a connected partition $\pi$, we have the identity
\[
NC(x;\pi)=\frac{1}{1-C(x;\pi)},
\]
and for arbitrary $\tau$, we have the identity
\[
 C(x;1[\tau]1)=\frac{x}{1-xNC(x;\tau)}.
\]

Combining these two identities and simplifying, we deduce that
\[
 NC(x;12321)= \frac{1-3x+x^2}{(1-x)(1-3x)}=1+\sum_{n\ge 1} \frac{3^{n-1}+1}{2}
x^n.
\]

Let us now turn to the pattern $\tau=12333$. The generating function for the empty partition together with those that have a single block is of course~$1/(1-x)$. On the other hand,
a non-crossing partition with at least two blocks avoids $\tau$ if and only if
it has a decomposition of the form $11\dotsb
12[\rho^1]2[\rho^2]2\dotsb2[\rho^k]1[\sigma^1]
1[\sigma^2]1\dotsb1[\sigma^{\ell}]$ for some $k\ge 1$ and $\ell\ge 0$, where
the $\rho^i$ and $\sigma^j$ are $111$-avoiding non-crossing partitions. It is
known that non-crossing $111$-avoiding partitions are counted by the Motzkin
numbers (\cite[sequence A001006]{oeis}), and their generating function satisfies
the identity
\[
 NC(x;111)=1+xNC(x;111)+(xNC(x;111))^2.
\]
We deduce that
\[
 NC(x;12333)=\frac{1}{1-x}+ \frac{x^2 NC(x;111)}{(1-x)(1-xNC(x;111))^2},
\]
from which the result easily follows.
\end{proof}

We remark that the counting function of $12333$-avoiding non-crossing
partitions (and therefore also $12321$-avoiding non-crossing partitions) has
been encountered before in different contexts (see \cite[sequence
A124302]{oeis}).

We may again use Theorem~\ref{thm-add1} iteratively to extend the equivalence
$12333\nceq 12321$ to an infinite sequence of equivalences.

\begin{corollary}\label{cor-sequence2}
For every $k\ge 3$, the pattern $12\cdots k(k+1)k(k-1)\cdots2$ is
nc-equivalent to $12\cdots(k-1)kkk(k-2)\cdots1$, and the pattern
$12\cdots(k-1)k(k-1)\cdots1$ is nc-equivalent to
$12\cdots kkk(k-2)\cdots2$.
\end{corollary}

As an application of the above theorems, one may completely identify the Wilf-equivalence classes corresponding to $\{1212,\tau\}$, where $\tau$ is of
size at most six. For example, we have the following equivalences in the cases
when $\tau$ is of size four or five, while the corresponding table for size
six can be found on the second author's webpage~\cite{web}.

\begin{itemize}
\item \mbox{\cite{MS1}} $1232\nceq1213\nceq1221\nceq1122$ by
Theorem~\ref{thm-permute}, Theorem~\ref{thm-1t1s} and Corollary~\ref{cor-add1},

\item $1123\nceq1223\nceq1233$ by Theorem~\ref{thm-permute},

\item \mbox{\cite{MS2}} $1211\nceq1121$ by Theorem~\ref{thm-conswap},

\item \mbox{\cite{MS3}} $1112\nceq1222$ by Theorem~\ref{thm-permute},

\item $1234$,

\item $1231$,

\item $1111$,
\end{itemize}

\begin{itemize}
\item
$12332\nceq12333\nceq12133\nceq11123\nceq12213\nceq12321\nceq12223\nceq11232$
by Theorem~\ref{thm-permute}, Theorem~\ref{thm-1t1s}, Corollary~\ref{cor-add1}
and Corollary~\ref{cor-sequence2},

\item $12234\nceq12334\nceq12344\nceq11234$ by Theorem~\ref{thm-permute},

\item $12113\nceq12322\nceq12232\nceq11213$ by Theorem~\ref{thm-permute} and
Theorem~\ref{thm-conswap},

\item $12233\nceq11233\nceq11223$ by Theorem~\ref{thm-permute},

\item $12343\nceq12134\nceq12324$ by Theorem~\ref{thm-permute},

\item $11122\nceq12221\nceq11222$ by Theorem~\ref{thm-permute} and
Theorem~\ref{thm-1t1s},

\item $11121\nceq11211\nceq12111$ by Theorem~\ref{thm-conswap},

\item $12331\nceq12231$ by Theorem~\ref{thm-subst2} using
Theorem~\ref{thm-permute},

\item $12342\nceq12314$ by Theorem~\ref{thm-permute},

\item $11231\nceq12311$ by Theorem~\ref{thm-conswap},

\item $12211\nceq11221$ by Theorem~\ref{thm-conswap},

\item $12222\nceq11112$ by Theorem~\ref{thm-permute},

\item $12131$,

\item $12341$,

\item $12345$,

\item $11111$.
\end{itemize}

\section{Avoiding two patterns of size four}\label{sec-44}

Let us say that a pair of patterns $(\sigma,\tau)$ is a \emph{$(4,4)$-pair}, if
$\sigma$ and $\tau$ are two distinct partitions of size four. In this section,
we will provide the full classification of equivalences among all $(4,4)$-pairs.

\subsection{Previously known equivalences}
Equivalences among sets of patterns have been previously studied in a series of
papers by Mansour and Shattuck~\cite{MS3,MS2,MS4,MS5,MS1}, and some of
the equivalence classes of $(4,4)$-pairs have been identified. Specifically, the
following results are known.

 \begin{theorem}[Theorem 1.1 in \cite{MS2}]
 If $n \geq 0$, then $p_n(u,v)=w_{n}$ for the following pairs $(u,v)$:
\[
(1211,1212), (1121,1212), (1121,1221),
(1112,1123), (1122,1123).
\]
 The generating function for the sequence $w_n$ is given by
 \begin{align*}
 \sum_{n \geq 0}w_nx^n=\frac{(1-x)^2-\sqrt{1-4x+2x^2+x^4}}{2x^2}.
 \end{align*}
 \end{theorem}

 \begin{theorem}[Theorem 1.1 in \cite{MS3}]
 If $n \geq 0$, then $p_n(u,v)=L_n$ for the following pairs $(u,v)$:
\[
(1222,1212), (1112,1212), (1211,1221), (1222,1221).
\]
 The generating function for the sequence $L_n$ is given by
 \begin{align*}
 \sum_{n\geq0}L_nx^n=\frac{1-3x+\sqrt{1-2x-3x^2}}{2(1-3x)}.
 \end{align*}
 \end{theorem}

Furthermore, results from \cite{MS1} imply the following fact.

\begin{fact}
 These pairs are all equivalent: $(1112,1213)$,
$(1122,1212)$, $(1123,1213)$, $(1123,1223)$, $(1211,
 1231)$, $(1212,1213)$,
 $(1221,1231)$, $(1222,1223)$, $(1222,1232)$, $(1212,1232)$, and $(1212,1221)$.
Moreover, for any such pair $(u,v)$,  $p_n(u,v)$ is equal to $F_{2n-2}$, where
$F_i$ is the $i$-th Fibonacci number.
\end{fact}

\subsection{Known results on pattern equivalences}

In the paper on partial patterns in matchings~\cite{pp}, the authors introduce the notion of strong partition equivalence.
We say that two patterns $\sigma$ and $\tau$ are \emph{strongly partition
equivalent}, if there exists a bijection $f$ between the sets of
$\sigma$-avoiding and $\tau$-avoiding partitions with the property that for any
$\sigma$-avoiding partition $\rho$, the number of blocks of $\rho$ is equal to
the number of blocks of $f(\rho)$, and moreover for any $i$, the $i$-th block of
$\rho$ has the same size as the $i$-th block of $f(\rho)$. Intuitively, strong
partition equivalence means that we can bijectively map $\sigma$-avoiders to
$\tau$-avoiders by just permuting the letters of their standard representation.

The concept of strong partition equivalence is first explicitly used
in~\cite{pp}, although most pairs of strongly partition equivalent patterns
follow from the bijections constructed in an earlier paper~\cite{parts}.
Let us list the known facts about strong partition equivalence (references
point to the corresponding statement in~\cite{parts}).

\begin{fact}[Lemma 9, Theorem 18, Corollary 18] \label{fac11}
For every $k$ and every partition $\tau$, the two partitions
$$12\dotsb k(\tau+k)12\dotsb k\mbox{ and  }12\dotsb k(\tau+k)k(k-1)\dotsb 1$$
are strongly partition equivalent.
\end{fact}

\begin{fact}[Theorem 31] Let $p_n(\rho;a_1,a_2,\dotsc,a_m)$ be the number of
partitions that avoid $\rho$ such that the $i$-th block has $a_i$ elements. Then
for any partition $\tau$,
\[
p_n(1(\tau+1);a_1,a_2,\dotsc,a_m)=\binom{n-1}{a_1-1}
p_{n-a_1}(\tau;a_2,a_3,\dotsc,a_m).
\]
Consequently, if $\tau$ and $\sigma$ are strongly partition-equivalent, then
so are $1(\tau+1)$ and $1(\sigma+1)$.
\end{fact}

\begin{fact}[Theorem 34]\label{fac33}
For every $j\geq1$ and $k\geq 0$, the partition $1^j21^k$ is strongly partition-equivalent
to
$1^{j+k}2$.
\end{fact}

\begin{fact}[Theorem 42]\label{fac66}
For every sequence $s$ over the alphabet $[m]$, for every $p\ge
1$ and $q\ge 0$, the partitions $$12\dotsb m (m+1)^{p}(m+2)(m+1)^q s\mbox{ and }
12\dotsb m (m+1)^{p+q}(m+2) s$$
are strongly partition-equivalent.
\end{fact}

\begin{fact}[Theorem 48]\label{fac77}
For every $k$, all the partitions of size $k$ that start with $12$ and that
contain two
occurrences of the symbol 1, one occurrence of the symbol 3, and all their
remaining symbols are equal to 2, are mutually strongly partition-equivalent.
\end{fact}

In the study of multi-avoidance, the concept of strong partition equivalence
becomes relevant through the following simple result.

\begin{theorem}\label{thm-strong}
 Let $\rho$ be a pattern of the form $12\dotsb (k-1)k^a$ for some $a\ge 1$ and
$k\ge 1$. That is, $\rho$ is formed by a strictly increasing sequence of length
$k$, followed by another $a-1$ occurrences of the symbol~$k$. Suppose that the
two patterns $\sigma$ and $\tau$ are strongly partition equivalent. Then the
two pattern pairs $(\rho,\sigma)$ and $(\rho,\tau)$ are equivalent.
\end{theorem}
\begin{proof}
 Note that a partition avoids $\rho$ if and only if it either has fewer than $k$
blocks, or for every $i\ge k$, its $i$-th block has size less than~$a$. In
other words, avoidance of $\rho$ can be characterized as a property of block
sizes.

Now assume that $\sigma$ and $\tau$ are strongly partition equivalent via a
bijection $f$. Since $f$ preserves the sizes of each block, we know that a
partition $\pi$ avoids $\rho$ if and only if $f(\pi)$ avoids~$\rho$. In
particular, $f$ maps the set of $(\rho,\sigma)$-avoiding
partitions bijectively to the set of $(\rho,\tau)$-avoiding partitions.
\end{proof}

\subsection{General arguments}

Before we deal with individual $(4,4)$-pairs, we first provide several general
results applicable to infinite families of pattern-avoiding classes.
Our first argument involves patterns containing one symbol equal to 2 and
the remaining symbols equal to 1. Fix such a pattern $\sigma=1^a21^b$ with
$a\ge 1$  and~$b\ge 0$. Let $k=a+b+1$ be the size of $\sigma$.

For a set of patterns $T$, let $T'$ denote the set $\{1(\tau+1),\tau\in T\}$.

For a set of patterns $R$, let $P_n(\sigma,R)$ denote the set of partitions of
size
$n$ that avoid the pattern $\sigma$ as well as all the patterns in $R$, and let
$P_n(\sigma,R;i)$ denote the set of partitions in $P_n(\sigma,R)$ whose first
block has
size~$i$. Let $f_n(\sigma,R)$ and $f_n(\sigma,R;i)$ denote the cardinality
of~$P_n(\sigma,R)$ and $P_n(\sigma,R;i)$, respectively.

\begin{lemma}\label{lem-1a21b} For any set of patterns $T$, and for $\sigma$
and $T'$ as above, we have
\[
f_n(\sigma,T')= \sum_{i=1}^{k-2}
f_{n-i}(\sigma,T)\binom{n-1}{i-1}+\sum_{i=k-1}^{n}
f_{n-i}(\sigma,T)\binom{n-i+k-3}{k-3}.
\]
\end{lemma}
\begin{proof}
We will compute the size of $P_n(\sigma,T';i)$ for $i\in[n]$. Suppose first that
$i<k-1$. Then a partition $\pi$ of size $n$ belongs to $P_n(\sigma,T';i)$ if and
only
if its first block has size $i$ and the remaining blocks induce a partition that
belongs to $P_{n-i}(\sigma,T)$. Therefore, we have
$f_n(\sigma,T';i)=f_{n-i}(\sigma,T)\binom{n-1}{i-1}$.

Now suppose that $i\ge k-1$. Then a partition $\pi$ of size $n$ belongs to
$P_n(\sigma,T';i)$ if and only if its first block has size $i$, the remaining
blocks induce a partition that belongs to $P_{n-i}(\sigma,T)$, and moreover, no
symbol
greater than 1 may have $a$ occurrences of $1$ before it and $b$ occurrences of
$1$ after it. That means that every symbol greater than 1 appears either before
the $a$-th occurrence of 1 or after the $(i-b+1)$-th occurrence of 1. That
gives $f_n(\sigma,T';i)=f_{n-i}(\sigma,T)\binom{n-i+k-3}{k-3}$.

Summing over all $i\in[n]$ gives the result.
\end{proof}

Note that the formula in the previous lemma does not depend on $a$ and~$b$. The
next statement is a direct consequence of Lemma~\ref{lem-1a21b}.
\begin{corollary}\label{cor-1a21b}
Let $\sigma$ and $\rho$ be two patterns of size $k$, with $\sigma=1^a21^b$
and $\rho=1^c21^d$. Let $T$ and $U$ be two sets of patterns. Let
$T'=\{1(\tau+1), \tau\in T\}$ and similarly for $U'$ and~$U$. If
$\{\sigma\}\cup T\sim \{\rho\}\cup U$, then $\{\sigma\}\cup T'\sim \{\rho\}\cup
U'$.
\end{corollary}

\begin{proposition}\label{P41}
If $T$ is a set of patterns and $T'$ is the set of patterns $1(\tau+1)$ where
$\tau \in T$, then for every $n\ge 1$, we have $p_n(T')=\sum_{k=0}^{n-1}
\binom{n-1}{k}p_k(T)$. Consequently, if $T$ and $R$ are sets of patterns
such that $T\sim R$, then $T' \sim R'$.
\end{proposition}
\begin{proof}
 It is enough to observe that a partition $\pi$ avoids $T'$ if and only if the
subpartition of $\pi$ obtained by removing the first block of $\pi$
avoids~$T$. Therefore, there are exactly $\binom{n-1}{k}p_k(T)$ partitions in
$P_n(T')$ whose first block has size $n-k$.
\end{proof}

\begin{corollary}\label{cor-1222}
 The $(4,4)$-pairs $(1222,1234)$ and $(1211,1234)$ are equivalent.
\end{corollary}
\begin{proof}
 We express $p_n(1222,1234)$ using Proposition~\ref{P41} and $p_n(1211,1234)$
using Lemma~\ref{lem-1a21b}. We have, for $n\ge 1$,
\begin{align*}
 p_n(1222,1234)&=\sum_{k=0}^{n-1}\binom{n-1}{k}p_k(111,123)\\
&=1+(n-1)+2\binom{n-1 } {2}+3\binom{n-1}{3}+3\binom{n-1}{4}
\end{align*}
and
\begin{align*}
 p_n(1211,1234)&=p_{n-1}(123,1211)+(n-1)p_{n-2}(123,1211)\\&+1+\sum_{k=1}^{n-3}
p_k(123,1211)(k+1).
\end{align*}
From Theorem~\ref{thm-123}, we deduce that $p_n(123,1211)=n+\binom{n-1}{2}$. Substituting
into the above expression and simplifying shows that both $p_n(1211,1234)$ and
$p_n(1222,1234)$ are equal to $(n^4-6n^3+19n^2-22n+16)/{8}$ for $n\ge 1$.
\end{proof}

Reasoning similar to that used in the proof of Lemma \ref{lem-1a21b} above
yields the following result, whose proof we omit.

\begin{proposition}\label{P43}
Let $T$ be a set of patterns and $T'$ be the set of patterns $1(\tau+1)$, where
$\tau \in T$. Let $a_n(R)$ (respectively, $b_n(R)$) denote the number
of partitions of $[n]$ that avoid all the patterns in $R$ as well as both $1112$
and $1121$ (respectively, $1121$ and $1211$). Then, for all $n\ge 4$, we have
$$a_n(T')=a_{n-1}(T)+(n-1)a_{n-2}(T)+a_{n-3}(T)+a_{n-4}(T)+\cdots+a_0(T) \text{ and} $$
$$b_n(T')=b_{n-1}(T)+(n-1)b_{n-2}(T)+b_{n-3}(T)+b_{n-4}(T)+\cdots+b_0(T).$$
\end{proposition}

\subsection{Specific patterns of size four}

In the next theorem, we consider the case of avoiding $1213$ and another
pattern of the form $1(\tau+1)$.

\begin{theorem}\label{P44}
Let $\rho=1(\tau+1)$ be any pattern of size at least two such that the
rightmost letter of $\tau$ is greater than~1. Then the
generating function $H_\rho(x)$ is given by
$$1+\frac{x}{1-x}H_\tau(x)+\frac{x^2}{(1-x)(1-2x)}(H_\tau(x)-1),$$
where $H_{\sigma}(x)$ denotes the generating function for the number of partitions of $[n]$
that avoid $1213$ and a pattern $\sigma$.
Moreover, if $(1213,\tau)\sim(1213,\tau')$, then
$(1213,1(\tau+1))\sim(1213,1(\tau'+1))$.
\end{theorem}
\begin{proof}
Let us write an equation for the generating function $H_\rho(x)$. For each nonempty
partition $\pi$ of $[n]$ that avoids $1213$ and $\rho$, either the first block of $\pi$
contains only $1$ or it contains $1$ and $2$ or it contains $1$ and at least one
other element, not $2$. The contributions from the first two cases are
$xH_\tau(x)$ and $x(H_\rho(x)-1)$, respectively.  Each partition $\pi$ in the last
case must have the form $12\pi'1\alpha$, where $\pi'$ does not contain $1$ and $\alpha$ is a (possibly empty) word on $\{1,2\}$, which
implies a contribution of $\frac{x^2}{1-2x}(H_\tau(x)-1)$. Hence,
$$H_\rho(x)=1-x+xH_\tau(x)+xH_\rho(x)+\frac{x^2}{1-2x}(H_\tau(x)-1),$$
which yields the required result.
\end{proof}

\begin{example}\label{P45}
We consider some specific examples.  Theorem \ref{P44}, together with the fact that $H_{112}(x)=\frac{1-x}{1-2x}$,
implies $$H_{1223}(x)=1+\frac{x(1-3x+3x^2)}{(1-2x)^2(1-x)}.$$
Since $112\sim123$, we then have
$(1213,1223)\sim(1213,1234)$, by Theorem \ref{P44}.

Using the same reasoning as in the proof of Theorem \ref{P44} gives
\begin{itemize}
\item $H_{1231}(x)=1+xH_{1231}(x)+x(H_{1231}(x)-1)+\frac{x^3}{(1-x)(1-2x)}$, which implies $H_{1231}(x)=H_{1223}(x),$ whence $(1213,1223)\sim(1213,1231);$
\item $H_{1221}(x)=1+xH_{1221}(x)+x(H_{1221}(x)-1)+\frac{x^3}{(1-x)^3}$, which implies $$H_{1221}(x)=\frac{1-4x+6x^2-3x^3+x^4}{(1-x)^3(1-2x)};$$
\item $H_{1232}(x)=1+xH_{121}(x)+x(H_{1232}(x)-1)+\frac{x^3}{(1-x)^2(1-2x)}$, whence $H_{1232}(x)=H_{1221}(x).$
\end{itemize}
Similarly, we obtain $$H_{1233}(x)=H_{1221}(x)\mbox{ and }H_{1121}(x)=\frac{(1-x-x^2)(1-x)^2}{1-4x+4x^2-2x^4}.$$
\end{example}

\begin{example}\label{P46}
Let $L_\sigma(x)$ be the generating function for the number of partitions of $[n]$ that avoid $1231$ and $\sigma$.
In this example, we study the generating function $L_\sigma(x)$ in a couple of cases.

First, we consider the case $\sigma=1121$. Note that each nonempty partition $\pi$ that avoids $1231$ and $1121$ can be expressed as $\pi=11\cdots1\pi'$, $\pi=122\cdots211\cdots1\pi''$, or $\pi=122\cdots211\cdots122\cdots2\pi'''$, where $\pi'$ does not contain the letter $1$, $\pi''$ starts with $3$ if nonempty (and is such that $22\cdots2\pi''$ is a nonempty partition on the letters $\{2,3,\ldots\}$), and $\pi'''$ does not contain the letters $1$ or $2$. Thus the generating function $L_{1121}(x)$ satisfies
$$L_{1121}(x)=1+\frac{x}{1-x}L_{1121}(x)+\frac{x^2}{1-x}(L_{1121}(x)-1)+\frac{x^4}{(1-x)^3}L_{1121}(x),$$
whence $L_{1121}(x)=\frac{(1-x-x^2)(1-x)^2}{1-4x+4x^2-2x^4}$.

Next, we consider the case $\sigma=1232$. From the structure, each nonempty partition $\pi$ that avoids $1231$ and $1232$ can be written as $\pi=1\pi'$, where $\pi'$ does not contain the letter $1$, or as $\pi=1\pi''$, where $\pi''$ is nonempty, or as $\pi=122\cdots21\alpha\pi'''$, where $\alpha$ is a (possibly empty) word in $\{1,2\}$ and $\pi'''$ does not contain the letters $1$ or $2$. Thus the generating function $L_{1232}(x)$ satisfies
$$L_{1232}(x)=1+xL_{121}(x)+x(L_{1232}(x)-1)+\frac{x^3}{(1-x)(1-2x)}L_{121}(x),$$
whence $L_{1232}(x)=1+\frac{x(1-3x+3x^2)}{(1-2x)^2(1-x)}$.
\end{example}

The remaining results in this section are of a more specific nature and cover
most of the equivalences in the table below left to be shown concerning the
avoidance of two patterns of size four.  If $m$ and $n$ are positive integers,
then $[m,n]=\{m,m+1,\ldots,n\}$ if $m \leq n$ and $[m,n]=\varnothing$ if $m>n$.

\begin{proposition}\label{PrA1}
If $n\geq1$, then 
\[p_n(1221,1232)=p_n(1221,1223)=1+(n-1)2^{n-2}.\]
\end{proposition}
\begin{proof}
Suppose $\pi \in P_{n,k}(1221,1232)$, where $k \geq 2$.  We consider whether or
not $\pi$ contains a repeated letter greater than one.  Thus, we see that $\pi$
must be of one of the following two forms:
\begin{enumerate}
\item $\pi=11^{a_1}21^{a_2}\cdots j1^{a_j}j^a(j+1)^{a_{j+1}}\cdots k^{a_k}$, where $2\leq j \leq k$, $a \geq 1$, $a_i\geq0$ if $i \in [j]$, and $a_i\geq 1$ if $i\in [j+1,k]$, 
\item $\pi=11^{b_1}21^{b_2}\cdots k1^{b_k}$, where $b_i\geq0$ if $i\in[k]$.
\end{enumerate}
This implies
\begin{align*}
p_n(1221,1232)&=1+\sum_{k=2}^n\sum_{j=2}^k\binom{n-k-1+k}{k}+\sum_{k=2}^n\binom{n-k+k-1}{k-1}\\
&=1+\sum_{k=2}^n(k-1)\binom{n-1}{k}+(2^{n-1}-1)=1+(n-1)2^{n-2}.
\end{align*}
For the second case, note that if $\pi \in P_{n,k}(1221,1223)$, where $k\geq 2$, then it must be of one of the following two forms:
\begin{enumerate}
\item $\pi=11^{a_1}21^{a_2}\cdots k1^{a_k}j^a$, where $2\leq j \leq k$, $a \geq 1$, and $a_i\geq0$ if $i \in [k]$,

\item $\pi=11^{b_1}21^{b_2}\cdots k1^{b_k}$, where $b_i\geq0$ if $i\in[k]$.
\end{enumerate}
This implies
\begin{align*}
p_n(1221,1223)&=1+\sum_{k=2}^{n}\sum_{a=1}^{n-k}(k-1)\binom{n-a-1}{k-1}+\sum_{k=2}^n\binom{n-1}{k-1}\\
&=1+\sum_{k=2}^n(k-1)\binom{n-1}{k}+(2^{n-1}-1)=1+(n-1)2^{n-2},
\end{align*}
which completes the proof.
\end{proof}

\begin{proposition}\label{PrAA1}
For $n \geq 1$, \[
p_n(1212,1123)=p_n(1212,1233)=p_n(1212,1223)=1+(n-1)2^{n-2}.
\]
\end{proposition}
\begin{proof}
From Theorem~\ref{thm-permute}, we know that the three pairs of patterns are equivalent. It is therefore enough to show that $p_n(1212,1223)=1+(n-1)2^{n-2}$. Suppose that $\pi \in P_{n,k}(1212,1223)$, where $k \geq 2$.  Then we either have $\pi=12\cdots (k-1)k^a\alpha$, where $a\geq1$ and $\alpha$ is a non-increasing word (i.e., it never strictly increases) on $[k-1]$, or
$$\pi=11^{b_1}21^{b_2}\cdots\ell 1^{b_\ell}(\ell+1)(\ell+2)\cdots (k-1)k^a\beta,$$
where $a\geq1$, $1 \leq \ell \leq k-1$, $b_\ell\geq1$, $b_i \geq 0$ if $1 \leq i < \ell$, and $\beta$ is a non-increasing word on $[\ell+1,k-1]\cup\{1\}$.  Summing over $a$ and $\ell$ implies that the members of $P_{n,k}(1212,1223)$ number
$$\sum_{a=1}^{n-k+1}\binom{n-a-1}{k-2}+\sum_{a=1}^{n-k}\sum_{\ell=1}^{k-1}\binom{n-a-1}{k-1}=\binom{n-1}{k-1}+(k-1)\binom{n-1}{k}.$$
Summing the last expression over $k \geq 2$ then gives $(n-1)2^{n-2}$ and completes the proof.
\end{proof}

\begin{proposition}\label{PrA2}
We have
$$\sum_{n\geq0}p_n(1122,1221)x^n=\frac{1-4x+5x^2-x^3}{(1-x)(1-2x)^2}.$$
\end{proposition}
\begin{proof}
If $\tau$ is a word, then let $f_\tau(x)$ denote the generating function for
the number of partitions (of size at least $|\tau|$) avoiding $\{1122,1221\}$
of the form $\tau\pi'$.  We first compute $f_\tau(x)$ in the case when
$\tau=12\cdots m1$, where $m\geq 2$.  Suppose $\pi$ is a member of
$P_n(1122,1221)$ whose first $m+1$ letters are $12\cdots m1$.  Then $\pi$ must
be of the form
$$\pi=1\alpha_01\alpha_12\alpha_23\alpha_3\cdots m\alpha_m,$$
where $\alpha_0=23\cdots m$, $\alpha_1$ is a (possibly empty) word in which all
the symbols different from $1$ form a strictly increasing sequence
$m+1,m+2,\ldots$, and $i\alpha_i$, $2 \leq i \leq m$, is a sequence which is
either empty or is nonempty and starts with $i$, with the letters of $\alpha_i$
comprising a (possibly empty) sequence of consecutive integers, the smallest of
which is one more than the largest letter occurring to the left of $\alpha_i$.
For example, if $m=4$ and $\pi=1234115116712894$, then $\alpha_0=234$,
$\alpha_1=1511671$, $2\alpha_2=289$, $3\alpha_3=\varnothing$ and $4\alpha_4=4$.
From the above decomposition, we see that
\begin{equation}\label{gena}
f_{12\cdots m1}(x)=\frac{x^{m+1}}{(1-x)^{m-1}(1-2x)}, \qquad m \geq 1.
\end{equation}
(Note a simpler argument applies to the $m=1$ case of \eqref{gena}.)

If $m\geq 1$, then let $\widetilde{f}_{12\cdots m}(x)$ denote the generating
function for the number of members $\pi=\pi_1\pi_2\cdots$ of $P_n(1122,1221)$ of
size at least $m$ such that $\pi_1\pi_2\cdots \pi_m=12\cdots m$, with
$\pi_{m+1}\leq m$ (if it occurs).  From the definitions and \eqref{gena}, we
have
\begin{align*}
\widetilde{f}_{12\cdots m}(x)&=x^m+\sum_{j=1}^m f_{12\cdots mj}(x)=x^m+\sum_{j=1}^m x^{j-1}f_{12\cdots (m-j+1)1}(x)\\
&=x^m+\sum_{j=1}^m \frac{x^{m+1}}{(1-x)^{m-j}(1-2x)}=\frac{x^m(1-x(1-x)^{m-1})}{(1-x)^{m-1}(1-2x)}.
\end{align*}
Thus, we have
\begin{align*}
\sum_{n\geq0}p_n(1122,1221)x^n&=1+\sum_{m\geq1}\widetilde{f}_{12\cdots m}(x)
=1+\sum_{m\geq1}\frac{x^m(1-x(1-x)^{m-1})}{(1-x)^{m-1}(1-2x)}\\&=\frac{1-4x+5x^2-x^3}{(1-x)(1-2x)^2},
\end{align*}
which completes the proof.
\end{proof}

From Proposition \ref{PrA2}, we see that $p_n(1122,1221)=1+(n-1)2^{n-2}$ for all $n \geq 1$.

We now consider the case of avoiding $\{1122,1223\}$.  For this, let us introduce the following notation.  Let $f_{a_1a_2\cdots a_m}=f_{a_1a_2\cdots a_m}(x)$ denote the generating function for the number of members $\pi=\pi_1\pi_2\cdots \pi_n \in P_n(1122,1223)$ where $n \geq m$ such that $\pi_1\pi_2\cdots \pi_m=a_1a_2\cdots a_m$, and let $f_{a_1a_2\cdots a_m}^*=f_{a_1a_2\cdots a_m}^*(x)$ denote the generating function counting the same partitions with the further restriction that $\max_{1\leq i\leq n}(\pi_i)=\max_{1\leq i\leq m}(a_i)$.  To establish the case $\{1122,1223\}$, we will need the following lemma.

\begin{lemma}\label{slem}
We have
\begin{equation}\label{sleme1}
F^{*}(x):=1+\sum_{k\geq1}f_{12\cdots k}^*(x)=\frac{1-4x+5x^2-2x^3+x^4}{(1-x)(1-2x)^2}.
\end{equation}
\end{lemma}
\begin{proof}
First note that if $k \geq 2$, then partitions $\pi$ enumerated by $f_{12\cdots k1}^*$ are of the form $$\pi=12\cdots k1\alpha_1\alpha_k\alpha_{k-1}\cdots\alpha_2,$$ where $\alpha_i$, $1 \leq i \leq k$, is either a nonempty word of the form $i1^a$ for some $a \geq0$ or is empty, whence $f_{12\cdots k1}^*=\frac{x^{k+1}}{(1-x)^k}$.
If $2 \leq j \leq k$, then within partitions enumerated by $f_{12\cdots kj}^*$, the letters $j+1,j+2,\ldots,k$ can appear only once (so as to avoid $1223$), whence
$$f_{12\cdots kj}^*=x^{k-j}f_{12\cdots jj}^*=x^{k-j+1}f_{12\cdots j}^*.$$
From the definitions, we then have
\begin{align}
f_{12\cdots k}^*&=x^k+f_{12\cdots k1}^*+\sum_{j=2}^kf_{12\cdots kj}^*\notag\\
&=x^k+\frac{x^{k+1}}{(1-x)^k}+\sum_{j=2}^kx^{k-j+1}f_{12\cdots j}^*, \qquad k \geq 1.\label{sleme2}
\end{align}
Summing \eqref{sleme2} over $k \geq 1$, and noting $f_1^*=\frac{x}{1-x}$, then yields
\begin{align*}
F^*(x)-1&=\frac{x}{1-x}+\frac{x^2}{1-2x}+\sum_{j\geq 2}f_{12\cdots j}^*(x)\sum_{k\geq j}x^{k-j+1}\\
&=\frac{x}{1-x}+\frac{x^2}{1-2x}+\frac{x}{1-x}\left(F^*(x)-\frac{x}{1-x}-1\right),
\end{align*}
which gives \eqref{sleme1}.
\end{proof}

We now prove the case $\{1122,1223\}$.

\begin{proposition}\label{PrA3}
We have
$$\sum_{n\geq0}p_n(1122,1223)x^n=\frac{1-4x+5x^2-x^3}{(1-x)(1-2x)^2}.$$
\end{proposition}
\begin{proof}
We first consider the generating function $\widetilde{f}_{12\cdots
k}=\widetilde{f}_{12\cdots k}(x)$, $k \geq 1$, for the number of partitions
$\pi=\pi_1\pi_2\cdots \pi_n \in P_n(1122,1223)$ having length at least $k$ such
that $\pi_1\pi_2\cdots\pi_k=12\cdots k$ and $\pi_{k+1}\leq k$ (if it exists).
From the definitions, we have
\begin{align}
\widetilde{f}_{12\cdots k}&=x^k+f_{12\cdots k1}+\sum_{j=2}^k f_{12\cdots kj} \notag\\
&=x^k+f_{12\cdots k1}+\sum_{j=2}^k x^{k-j+1}f_{12\cdots j}^*, \qquad k \geq 1, \label{PrA3e1}
\end{align}
for if $2 \leq j \leq k$, then $f_{12\cdots kj}=x^{k-j+1}f_{12\cdots j}^*$, since the letters $j+1, j+2, \ldots, k$ can only appear once in a partition enumerated by $f_{12\cdots kj}$, with no letters greater than $k$ occurring.  Furthermore, we have
\begin{equation}\label{PrA3e2}
f_{12\cdots k1}=\frac{x^{k+1}}{(1-x)^{k-1}(1-2x)}, \qquad k \geq 1,
\end{equation}
the enumerated partition $\pi$ having the form
$$\pi=1\alpha_01\alpha_1\alpha_k\alpha_{k-1}\cdots \alpha_2,$$
where $\alpha_0=23\cdots k$, $\alpha_1$ is a (possibly empty) word obtained by replacing the $2$'s occurring in a word in $\{1,2\}$, successively, with the letters $k+1,k+2,\ldots$, and $\alpha_i$, $2 \leq i \leq k$, is a sequence which is either empty or is nonempty and of the form $i1^a$ for some $a \geq 0$.

Summing \eqref{PrA3e1} over $k\geq 1$, and using \eqref{PrA3e2}, then gives
\begin{align*}
\sum_{n\geq0}p_n(1122,1223)x^n&=1+\sum_{k\geq 1}\widetilde{f}_{12\cdots k}(x)\\
&=\frac{1}{1-x}+\sum_{k\geq 1}\frac{x^{k+1}}{(1-x)^{k-1}(1-2x)}+\sum_{j\geq 2}f_{12\cdots j}^*(x)\sum_{k\geq j}x^{k-j+1}\\
&=\frac{1}{1-x}+\frac{x^2(1-x)}{(1-2x)^2}+\frac{x}{1-x}\left(F^*(x)-\frac{x}{1-x}-1\right),
\end{align*}
which yields the requested result, by \eqref{sleme1}.
\end{proof}

\begin{proposition}\label{PrB}
For $n \geq 0$, $p_n(1123,1233)=p_n(1231,1233)=p_n(1123,1232)$.  In
addition, we have
$$\sum_{n\geq0}p_n(1122,1232)x^n=\frac{1-4x+6x^2-3x^3+x^4}{(1-x)^3(1-2x)}.$$
\end{proposition}
\begin{proof}
To show the first statement, we define bijections between the classes which preserve the number of blocks $k$ as follows.  We may assume $k \geq 3$, since all of the patterns contain at least one $3$.  If $\pi \in P_{n,k}(1231,1233)$, then $\pi$ is necessarily of the form $\pi=1^a2\alpha32^b45\cdots k$, where $a \geq 1$, $b \geq 0$, and $\alpha$ is a word in $\{1,2\}$.  Let $\pi'=12\cdots (k-1)2^{a-1}1^bk\alpha$ and $\pi''=12\cdots (k-1)^{b+1}1^{a-1}k\alpha'$, where $\alpha'$ is obtained from $\alpha$ by replacing each occurrence of the letter $2$ with $k$.  Then it may be verified that the mappings $\pi \mapsto \pi'$ and $\pi \mapsto \pi''$ are bijections from $P_{n,k}(1231,1233)$ to $P_{n,k}(1123,1233)$ and $P_{n,k}(1123,1232)$, respectively.

To prove the second statement, suppose $\pi \in P_{n}(1122,1232)$ has at least two blocks.  Then $\pi$ necessarily has one of the following four forms, where $k$ denotes the number of blocks in the last three cases:
\begin{enumerate}
\item $\pi=1\pi'$, where $\pi'$ contains no $1$'s, avoids $\{1122,121\}$ and is nonempty.

\item $\pi=1^{r_1}21^{r_2}\cdots k1^{r_k}$, where $r_1\geq2$ and $r_i\geq0$ if $i \geq 2$.

\item $\pi=1\pi'1^{s_0}(\ell+1)1^{s_1}\cdots k1^{s_{k-\ell}}$, where $2 \leq \ell \leq k$, $s_0 \geq 1$, $s_i \geq0$ if $i\geq1$, and $\pi'$ is a nonempty partition on the letters $\{2,3,\ldots,\ell\}$ avoiding $\{1122,121\}$.

\item $\pi=12\cdots(\ell-1)\ell^a1^{t_0}\ell1^{t_1}\cdots k1^{t_{k-\ell+1}}$, where $2 \leq \ell \leq k$, $a \geq 1$, $t_0\geq 1$, $t_i\geq0$ if $i\geq 1$, and an $\ell$ appears after the second run of the letter $1$.
\end{enumerate}

Thus, we get
\begin{align*}
\sum_{n\geq0}p_n(1122,1232)x^n&=\frac{1}{1-x}+x(g(x)-1)+\frac{x^3}{(1-x)(1-2x)}
\\ &+\frac{x^2}{1-x}(g(x)-1)\left(1+\frac{x}{1-2x}\right)
+\frac{x^4}{(1-x)^3(1-2x)},
\end{align*}
where $g(x)=\sum_{n\geq0}p_n(1122,121)x^n$.  By the table in Section 2.4 above, we have $g(x)=\sum_{i=0}^3\left(\frac{x}{1-x}\right)^i$, and substituting this into the last equation gives the second statement.
\end{proof}

\begin{proposition}\label{PrC}
If $n \geq 2$, then $p_n(1123,1234)=p_n(1122,1233)=2^{n-5}(n^2-n+14)$.
\end{proposition}
\begin{proof}
We enumerate both classes directly.  If $n \geq 2$, then there are $2^{n-1}$ members of $P_n(1123,1234)$ having no $3$'s as well as $2^{n-1}$ members of $P_n(1122,1233)$ having at most one occurrence of the letter $2$.  So it remains to show that the members of $P_n(1123,1234)$ in which $3$ occurs as well as the members of $P_n(1122,1233)$ in which $2$ occurs at least twice both number $2^{n-5}(n-2)(n+1)$.

To show the first part, let $j$ denote the \emph{largest} index such that the $(j+2)$-nd letter of $\pi \in P_n(1123,1234)$ is $3$.  Then the letters to the right of position $j+2$ may constitute any binary word, while those to the left (excepting the first two, which are necessarily $1,2$) must form a non-decreasing binary word.  Upon conditioning further on the number, $\ell+1$, of $3$'s, we see that the remaining members of $P_n(1123,1234)$ number
\begin{align*}
\sum_{j=1}^{n-2}2^{n-2-j}\sum_{\ell=0}^{j-1}(j-\ell)\binom{j-1}{\ell}&=\sum_{j=1}^{n-2}2^{n-2-j}(j+1)2^{j-2}\\
&=2^{n-4}\sum_{j=1}^{n-2}(j+1)=2^{n-5}(n-2)(n+1),
\end{align*}
as required.

Members of $P_n(1122,1233)$ in which $2$ occurs at least twice must start
$1,2$, with any letters greater than $2$ occurring once.  The remaining
positions are to be filled by $1$'s and $2$'s in such a way that the word
comprising the letters in these positions avoids $122$ and contains at least one
$2$.  Let $a_m$ denote the number of words over $\{1,2\}$ of length $m$ avoiding
$122$
such that $2$ occurs at least once.  It may be verified that
$a_m=\binom{m+1}{2}$.  Upon conditioning on the number, $\ell$, of elements
greater than two occurring, we see that the remaining members of
$P_n(1122,1233)$ number
\begin{align*}
\sum_{\ell=0}^{n-3}\binom{n-2}{\ell}a_{n-2-\ell}&=\sum_{\ell=1}^{n-2}\binom{n-2}
{\ell}a_\ell\\
&=\frac{1}{2}\sum_{\ell=1}^{n-2}\binom{n-2}{\ell}(\ell^2+\ell)
=2^{n-5}(n-2)(n+1),
\end{align*}
which completes the proof.
\end{proof}

\subsection{The cases $\{1123,1211\}$ and $\{1123,1222\}$}

In order to establish these cases, we first consider a statistic on the members
of $P_n$ having at least two blocks as follows.  We will say that an
\emph{ascent} occurs at position $i$ in $\pi=\pi_1\pi_2\cdots \pi_n \in P_n$ if
$\pi_i<\pi_{i+1}$, where $1 \leq i \leq n-1$.  For example, the partition
$\pi=1223413142 \in P_{10}$ has ascents at positions $1$, $3$, $4$, $6$, and
$8$.  By the \emph{final ascent}, we will mean the \emph{largest index} $i$ such
that $\pi_i<\pi_{i+1}$.

\begin{definition}\label{df}
Suppose that the final ascent in $\pi=\pi_1\pi_2\cdots \pi_n \in P_n$ occurs at position $m$. Then let $fasc$ denote the statistic on $P_n$ defined by $fasc(\pi)=n-m+1$.
\end{definition}

For example, if $\pi=123241355311 \in P_{12}$, then the final ascent occurs at position $7$ and $fasc(\pi)=12-7+1=6$.  Note that $2\leq fasc(\pi) \leq n-k+2$ for all $\pi \in P_{n,k}$, where $n \geq k \geq 2$, the minimum being achieved by any partition whose last two entries form an ascent and the maximum achieved by partitions of the form $12\cdots k \tau$, where $\tau$ is any non-increasing $k$-ary word of length $n-k$.  If $n \geq k \geq 2$ and $2 \leq t \leq n-k+2$, then let
$$a_{n,k,t}=|\{\pi \in P_{n,k}(1123,1211):~fasc(\pi)=t\}|,$$
and define $a_{n,k,t}$ to be zero otherwise.  For example, we have $a_{4,3,2}=2$ since there are two members of $P_{4,3}(1123,1211)$ having $fasc$ value $2$, namely $1223$ and $1213$, and $a_{4,3,3}=3$, the partitions in this case being $1231$, $1232$, and $1233$. Furthermore, let $A_{n,k,t}$ denote the subset of partitions enumerated by $a_{n,k,t}$.  In the following lemma, we provide an explicit recurrence satisfied by the numbers $a_{n,k,t}$.

\begin{lemma}\label{lma}
If $3 \leq k \leq n$ and $2 \leq t \leq n-k+2$, then
\begin{equation}\label{lmae1}
a_{n,k,t}=a_{n-1,k-1,t}+\sum_{j=t-1}^{n-k+1}a_{n-2,k-1,j},
\end{equation}
with $a_{2,2,2}=1$ and for $n \geq 3$,
\begin{equation}\label{lmae2}
a_{n,2,t}= \left\{ \begin{array}{cl}
n-t+1, & \textrm{if \quad $3\leq t \leq n-1$;}\\
n-2, & \textrm{if \quad $t=2$;}\\
2, & \textrm{if \quad $t=n$.}
\end{array} \right.
\end{equation}
\end{lemma}
\begin{proof}
First note that members of $A_{n,k,t}$, where $3 \leq k \leq n$ and $2 \leq t \leq n-k+2$, have either one or two $1$'s, since such partitions must start with a single $1$ and can have at most a single additional $1$ coming after a letter greater than one.  To show \eqref{lmae1}, first observe that there are $a_{n-1,k-1,t}$ members of $A_{n,k,t}$ containing a single $1$, for writing a $1$ in front of $\alpha \in A_{n-1,k-1,t}$ (on the letters $\{2,3,\ldots,k\}$) does not affect the $fasc$ value.

If two $1$'s occur in $\pi \in A_{n,k,t}$, then $\pi$ may be formed from $\beta=\beta_1\beta_2\cdots \beta_{n-2} \in A_{n-2,k-1,j}$, where $t-1\leq j \leq n-k+1$, by first writing a $1$ in front of $\beta$ and then considering cases on $j$.  If $j \geq t$, then write a second $1$ just before the $(t-1)$-st letter of $\beta$ from the right to obtain $\pi=1\beta_1\beta_2\cdots\beta_{n-t+1}1\beta_{n-t}\cdots \beta_{n-2}$, and if $j=t-1$, then write a second $1$ at the end of $\beta$ to obtain $\pi=1\beta_1\beta_2\cdots \beta_{n-2}1$.  Note that $j \geq t$ in the first case implies no occurrence of $1123$ is introduced by the insertion of the second $1$ since the final $j-1 \geq t-1$ letters of $\beta$ form a non-increasing word.  Also, in the case when $j=t-1$, the addition of $1$ at the end increases the $fasc$ value by one.  Thus, the sum in \eqref{lmae1} conditions on the $fasc$ value of the partition resulting from the removal of the $1$'s from $\pi \in A_{n,k,t}$ containing
  two $1$'s.

We now turn to the case when $k=2$.  Note first that members of $A_{n,2,2}$, $n\geq 3$, are of the form $1\rho12$, where $\rho=1^a2^b$ with $a,b \geq 0$ and $a+b=n-3$, whence they number $n-2$.  There are two members of $A_{n,2,n}$, namely, $12\cdots2$ and $12\cdots 21$.  Finally, if $3 \leq t \leq n-1$, then $A_{n,2,t}$ is the set of size $n-t+1$ comprising the partitions $1^{n-t+1}2^{t-2}1$ and $1^{n-t+1}2^{t-1}$, together with the partitions of the form $1\rho12^{t-1}$, where $\rho=1^a2^b$ for some $a\geq 0$ and $b \geq 1$ satisfying $a+b=n-t-1$.  This establishes \eqref{lmae2} and completes the proof.
\end{proof}

We can now enumerate the partitions avoiding $\{1123,1211\}$.

\begin{theorem}\label{tm1}
We have

{\footnotesize 
\[\sum_{n\geq 0}
p_n(1123,1211)t^n=\frac{(1-t^2)\sqrt{(1-t)(1-t-4t^2)}}{2t^2(1-3t+t^2)}-\frac{
1-3t-2t^2+14t^3-15t^4+3t^5}{2t^2(1-t)^2(1-3t+t^2)}.
\]
}
\end{theorem}
\begin{proof}
First define the distribution polynomial $A_{n,k}(w)=\sum_{t=2}^{n-k+2}a_{n,k,t}w^t$ if $n \geq k \geq 2$, with $A_{n,k}(w)=0$ otherwise.  Multiplying \eqref{lmae1} above by $w^t$, summing over $2 \leq t \leq n-k+2$, and interchanging summation yields
\begin{align}
A_{n,k}(w)&=A_{n-1,k-1}(w)+\sum_{j=1}^{n-k+1}a_{n-2,k-1,j}\sum_{t=2}^{j+1}w^t \notag \\
&=A_{n-1,k-1}(w)+\frac{w^2}{1-w}\sum_{j=2}^{n-k+1}a_{n-2,k-1,j}(1-w^j)\notag\\
&=A_{n-1,k-1}(w)+\frac{w^2}{1-w}\left(A_{n-2,k-1}(1)-A_{n-2,k-1}(w)\right), \qquad n\geq k \geq 3, \label{tm1e1}
\end{align}
with $A_{2,2}(w)=w^2$ and
\begin{equation}\label{tme2}
A_{n,2}(w)=(n-2)w^2+\sum_{t=3}^{n-1}(n-t+1)w^t+2w^n, \qquad n \geq 3.
\end{equation}

Next define $A_n(v,w)=\sum_{k=2}^n A_{n,k}(w)v^k$ if $n \geq 2$, with $A_1(v,w)=A_0(v,w)=0$.  Multiplying \eqref{tm1e1} by $v^k$ and summing over $3 \leq k \leq n$ yields the recurrence
\begin{equation}\label{tme3}
A_n(v,w)-A_{n,2}(w)v^2=vA_{n-1}(v,w)+\frac{w^2v}{1-w}(A_{n-2}(v,1)-A_{n-2}(v,w)), \qquad n \geq 3.
\end{equation}
Define the generating function $A(t,v,w)=\sum_{n\geq 2}A_n(v,w)t^n$. 
Multiplying \eqref{tme3} by $t^n$ and summing over all $n \geq 3$ gives
\begin{multline*}
A(t,v,w)-A_2(v,w)t^2-v^2\sum_{n \geq
3}A_{n,2}(w)t^n\\=vtA(t,v,w)+\frac{w^2vt^2}{1-w}(A(t,v,1)-A(t,v,w)),
\end{multline*}
which implies by \eqref{tme2} the relation
\begin{equation}\label{tme4}
\left(1-vt+\frac{w^2vt^2}{1-w}\right)A(t,v,w)=\frac{w^2v^2t^2(1-t(1-w)(1-t))}{(1-t)^2(1-tw)}+\frac{w^2vt^2}{1-w}A(t,v,1).
\end{equation}
This type of functional equation may be solved systematically using the \emph{kernel method} (see \cite{BBD}).  Setting the coefficient of $A(t,v,w)$ equal to zero and solving for $w=w_0$ in terms of $v$ and $t$ yields
\begin{equation}\label{tme5}
w_0=\frac{1-vt-\sqrt{(1-vt)(1-vt-4vt^2)}}{2vt^2}.
\end{equation}
Note that of the two possible values of $w_0=w_0(t,v)$, only this one yields a power series in $t$ and $v$.  Setting $w=w_0$ in \eqref{tme4} then gives
\begin{equation}\label{tme6}
A(t,v,1)=\frac{v(w_0-1)(1-t(1-w_0)(1-t))}{(1-t)^2(1-tw_0)}.
\end{equation}
Note that $A(t,v,1)$ is the generating function for the cardinality of
$P_{n,k}(1123,1211)$ for $n \geq k \geq 2$.  Letting $v=1$, and including the
$k=0$ and $k=1$ cases, we see that

{\footnotesize
\begin{align*}
\sum_{n \geq 0}p_n(1123,1211)t^n&=\frac{1}{1-t}+A(t,1,1)\\
&=\frac{(1-t^2)\sqrt{(1-t)(1-t-4t^2)}}{2t^2(1-3t+t^2)}-\frac{1-3t-2t^2+14t^3-15t^4+3t^5}{2t^2(1-t)^2(1-3t+t^2)},
\end{align*}
}

\noindent by \eqref{tme5} and \eqref{tme6}, as desired.
\end{proof}

\emph{Remark:}~~Note that $w_0=C\left(\frac{vt^2}{1-vt}\right)$, where
$C(x)=\frac{1-\sqrt{1-4x}}{2x}=\sum_{n\geq 0}c_nx^n$ is the generating function
for the Catalan number $c_n=\frac{1}{n+1}\binom{2n}{n}$.

We now turn to the case of avoiding $\{1123,1222\}$.  For this, we first consider the avoidance of $\{1123,111\}$.  If $n \geq k \ge2$ and $2 \leq t \leq n-k+2$, then let $$b_{n,k,t}=|\{\pi \in P_{n,k}(1123,111):~fasc(\pi)=t\}|,$$
and define $b_{n,k,t}$ to be zero otherwise.  For example, we have $b_{4,3,2}=2$ (the enumerated partitions being $1213$ and $1223$) and $b_{5,3,3}=4$ (the partitions being $12132$, $12133$, $12231$ and $12233$).  Note that $b_{n,k,t}=0$ for all $t$ if $n>2k$.

Next, let $$c_{n,k,t}=|\{\pi=\pi_1\pi_2\cdots \pi_n \in P_{n,k}(1123,1222):~fasc(\pi)=t \text{~and~} \pi_n\neq1\}|,$$
if $n \geq k \ge2$ and $2 \leq t \leq n-k+2$ and put $c_{n,k,t}=0$ otherwise.  For example, we have $c_{4,3,2}=2$ (for $1213$ and $1223$) and $c_{5,3,3}=3$ (for $12132$, $12133$, and $12233$).  In the following lemma, we provide a recurrence for $b_{n,k,t}$ and a relation between $c_{n,k,t}$ and $b_{n,k,t}$.

\begin{lemma}\label{lmb}
If $3 \leq k \leq n$ and $2 \leq t \leq n-k+2$, then
\begin{equation}\label{lmbe1}
b_{n,k,t}=b_{n-1,k-1,t}+\sum_{j=t-1}^{n-k+1}b_{n-2,k-1,j}
\end{equation}
and
\begin{equation}\label{lmbe2}
c_{n,k,t}=b_{n-1,k-1,t}+\sum_{j=t}^{n-k+1}c_{n-1,k,j}.
\end{equation}
If $k=2$, we have $b_{n,2,t}=0$ if $n \geq 5$, $b_{4,2,4}=b_{4,2,3}=b_{4,2,2}=1$, $b_{3,2,3}=2$, and $b_{3,2,2}=b_{2,2,2}=1$.  Furthermore,
\begin{equation}\label{lmbe3}
c_{n,2,t}= \left\{ \begin{array}{cl}
n-2, & \textrm{if \quad $t=2$;}\\
1, & \textrm{if \quad $t=3$;}\\
0, & \textrm{if \quad $t\geq 4$,}
\end{array} \right.
\end{equation}
if $n\geq 3$, with $c_{2,2,2}=1$.
\end{lemma}
\begin{proof}
The proof of \eqref{lmbe1} is similar to the one given above for \eqref{lmae1}; note that a partition is restricted to containing one or two $1$'s since we are now avoiding $111$, which also implies $b_{n,2,t}=0$ if $n \geq 5$.  One may verify directly the other conditions for $b_{n,2,t}$.

Let $C_{n,k,t}$ denote the subset of partitions counted by $c_{n,k,t}$.  To show \eqref{lmbe2}, first note that if $n \geq k \geq 3$, then members of $C_{n,k,t}$ containing a single $1$ are of the form $1\pi'$, where $\pi'$ is a partition on the letters $\{2,3,\ldots,k\}$ avoiding $\{1123,111\}$.  On the other hand, one may add a $1$ just before the $(t-1)$-st letter from the right of $\alpha=\alpha_1\alpha_2\cdots \alpha_{n-1} \in C_{n-1,k,j}$, $t \leq j \leq n-k+1$, to obtain $\pi=\alpha_1\alpha_2\cdots \alpha_{n-t}1\alpha_{n-t+1}\cdots \alpha_{n-1} \in C_{n,k,t}$ having two or more $1$'s.  Note that $j \geq t$ implies $\alpha_{n-t+1}\cdots \alpha_{n-1}$ is a non-decreasing word and thus contains no $1$'s since $\alpha_{n-1}>1$. Therefore, the sum on the right side of \eqref{lmbe2} is seen to count the members $\pi \in C_{n,k,t}$ containing two or more $1$'s by conditioning on the $fasc$ value of the partition resulting when we remove the right-most $1$ from $\pi$.

We now consider the case $k=2$.  First, note that there are $n-2$ members of $C_{n,2,2}$, $n \geq 3$, since they are of the form $1\rho12$, where $\rho$ is all $1$'s except for possibly a single occurrence of $2$.  There is a single member of $C_{n,2,3}$ given by $1^{n-2}22$.  Finally, there are no members of $C_{n,2,t}$ if $t \geq 4$ since the last letter not $1$ implies only $2$'s may follow the last ascent and these would then number $t-1 \geq 3$, which is not permitted.  This establishes \eqref{lmbe3} and completes the proof.
\end{proof}

Define the distribution polynomial $B_{n,k}(w)=\sum_{t=2}^{n-k+2}b_{n,k,t}w^t$
if $n \geq k \geq 2$.  When $k=2$, note that $B_{2,2}(w)=w^2$,
$B_{3,2}(w)=w^2+2w^3$, $B_{4,2}(w)=w^2+w^3+w^4$, and $B_{n,2}(w)=0$ if $n \geq
5$.  If $n \geq 2$, then let $B_n(v,w)=\sum_{k=2}^n B_{n,k}(w)v^k$, with
$B_0(v,w)=B_1(v,w)=0$.  Finally, define the generating function $B(t,v,w)$ by
$$B(t,v,w)=\sum_{n\geq 2}B_n(v,w)t^n.$$
Note that $B(t,v,w)$ counts all of the members of $P_n(1123,111)$ according to
the number of blocks and the \emph{fasc} value when $n\geq 3$, counts only the
partition $\{1\},\{2\}$ when $n=2$, and counts neither $\{1\}$ nor the empty
partition.  There is the following explicit formula for $B(t,v,w)$.

\begin{lemma}\label{lm1}
We have
\begin{align*}
B(t,v,w)&=\frac{w^2(vt^2-1)\sqrt{(1-vt)(1-vt-4vt^2)}}{2vt^2((1-w)(1-vt)+w^2vt^2)}\\
&~~+\frac{w^2(1-vt-3vt^2+v^2t^3+2v^3(1-w)t^4+2w^3(1-2w)t^5-2v^3w^3t^6)}{2vt^2((1-w)(1-vt)+w^2vt^2)}.
\end{align*}
\end{lemma}
\begin{proof}
Multiplying \eqref{lmbe1} by $w^t$ and summing over $2 \leq t \leq n-k+2$ implies
\begin{equation}\label{lm1e1}
B_{n,k}(w)=B_{n-1,k-1}(w)+\frac{w^2}{1-w}(B_{n-2,k-1}(1)-B_{n-2,k-1}(w)), \qquad n \geq k \geq 3,
\end{equation}
and multiplying \eqref{lm1e1} by $v^k$ and summing over $3 \leq k \leq n$ yields
\begin{equation}\label{lm1e2}
B_n(v,w)-B_{n,2}(w)v^2=vB_{n-1}(v,w)+\frac{w^2v}{1-w}(B_{n-2}(v,1)-B_{n-2}(v,w)), \qquad n \geq 3,
\end{equation}
which is seen to hold for $n=2$ as well, since $B_0(v,w)=B_1(v,w)=0$. 
Multiplying \eqref{lm1e2} by $t^n$ and summing over $n \geq 2$ then gives
\begin{multline}\label{lm1e3}
B(t,v,w)-w^2v^2t^2(1+(1+2w)t+(1+w+w^2)t^2)\\=vtB(t,v,w)+\frac{w^2vt^2}{1-w}(B(t,
v ,1)-B(t,v,w)).
\end{multline}
Letting $w=w_0$ in \eqref{lm1e3}, where $w_0$ is given by \eqref{tme5} above, implies
\begin{equation}\label{lm1e4}
B(t,v,1)=\frac{w_0^2v^2t^2(1+(1+2w_0)t+(1+w_0+w_0^2)t^2)}{1-vt}.
\end{equation}
Substituting \eqref{lm1e4} into \eqref{lm1e3}, solving for $B(t,v,w)$, and simplifying then yields the required expression.
\end{proof}

We can now enumerate the partitions avoiding $\{1123,1222\}$.

\begin{theorem}\label{tm2}
We have

{\footnotesize
\[
\sum_{n\geq 0}
p_n(1123,1222)t^n=\frac{(1-t^2)\sqrt{(1-t)(1-t-4t^2)}}{2t^2(1-3t+t^2)}-\frac{
1-3t-2t^2+14t^3-15t^4+3t^5}{2t^2(1-t)^2(1-3t+t^2)}.\]
}
\end{theorem}
\begin{proof}
Let $C_{n,k}(w)=\sum_{t=2}^{n-k+2}c_{n,k,t}w^t$.  Multiplying \eqref{lmbe2} by
$w^t$ and summing over $2 \leq t \leq n-k+2$ yields
\begin{multline}
C_{n,k}(w)=B_{n-1,k-1}(w)+\frac{w^2}{1-w}C_{n-1,k}(1)\\-\frac{w}{1-w}C_{n-1,k}
(w) , \qquad n \geq k \geq 3,\label{tm2e1}
\end{multline}
with $C_{2,2}(w)=w^2$ and $C_{n,2}(w)=(n-2)w^2+w^3$ if $n\geq 3$.  Let $C_n(v,w)=\sum_{k=2}^n C_{n,k}(w)v^k$ if $n \geq 2$.  Multiplying $\eqref{tm2e1}$ by $v^k$ and summing over $3 \leq k \leq n$ yields
\begin{align}
C_n(v,w)-C_{n,2}(w)v^2&=vB_{n-1}(v,w)+\frac{w^2}{1-w}(C_{n-1}(v,1)-C_{n-1,2}(1)v^2)\label{tm2e2}\\
&~~~-\frac{w}{1-w}(C_{n-1}(v,w)-C_{n-1,2}(w)v^2), \qquad n \geq 3. \notag
\end{align}

Let $C(t,v,w)=\sum_{n\geq 2}C_n(v,w)t^n$ and $\widetilde{C}(t,v,w)=v^2\sum_{n \geq 2}C_{n,2}(w)t^n$.
Note that
\begin{align}
\widetilde{C}(t,v,w)&=w^2v^2t^2+\frac{w^3v^2t^3}{1-t}+w^2v^2\sum_{n \geq 3}(n-2)t^n \notag\\
&~~~=\frac{w^2v^2t^2(1-t(1-w)(1-t))}{(1-t)^2}. \label{tm2e3}
\end{align}
Multiplying \eqref{tm2e2} by $t^n$ and summing over $n \geq 3$ gives
\begin{align}
C(t,v,w)-\widetilde{C}(t,v,w)&=vtB(t,v,w)+\frac{w^2t}{1-w}(C(t,v,1)-\widetilde{C}(t,v,1))\label{tm2e4}\\
&~~~-\frac{wt}{1-w}(C(t,v,w)-\widetilde{C}(t,v,w)). \notag
\end{align}
Substituting $w=\frac{1}{1-t}$ into \eqref{tm2e4}, and rearranging, implies
\begin{equation}\label{tm2e5}
C(t,v,1)=\widetilde{C}(t,v,1)+vt(1-t)B\left(t,v,\frac{1}{1-t}\right).
\end{equation}
By Lemma \ref{lm1},

{\footnotesize
\begin{multline}
B\left(t,v,\frac{1}{1-t}\right)=\frac{(1-vt^2)\sqrt{(1-vt)(1-vt-4vt^2)}}{
2vt^3(1-t-2vt+vt^2)}\label{tm2e6}\\
\qquad-\frac{1-(3+v)t+3t^2-(1-6v-v^2)t^3-v(8+3v)t^4+(3+3v-4v^2)t^5-v^2(1-2v)t^6}
{2vt^3(1-t)^3(1-t-2vt+vt^2)}.
\end{multline}
}

Note that $C(t,1,1)$ is the generating function for the number of elements of
$P_n(1123,1222)$, $n \geq 2$, ending in a letter greater than one, which implies
that $\frac{1}{1-t}C(t,1,1)$ counts all of the members of $P_n(1123,1222)$
having at least two blocks, for adding a string of $1$'s of arbitrary length to
the end does not otherwise affect the enumeration.  Thus, we have

{\footnotesize
\begin{align*}
\sum_{n\geq0}p_n(1123,1222)t^n&=\frac{1}{1-t}+\frac{1}{1-t}C(t,1,1)\\
&=\frac{1}{1-t}+\frac{1}{1-t}\widetilde{C}(t,1,1)+tB\left(t,1,\frac{1}{1-t}\right)\\
&=\frac{(1-t^2)\sqrt{(1-t)(1-t-4t^2)}}{2t^2(1-3t+t^2)}-\frac{1-3t-2t^2+14t^3-15t^4+3t^5}{2t^2(1-t)^2(1-3t+t^2)},
\end{align*}
}

\noindent as required, by \eqref{tm2e3}, \eqref{tm2e5}, and \eqref{tm2e6}.
\end{proof}

Comparison of Theorems \ref{tm1} and \ref{tm2} reveals that the pairs $\{1123,1211\}$ and $\{1123,1222\}$ are equivalent.  We have searched for a direct
bijection demonstrating this fact and leave it to the reader as an open problem.

\subsection{Table of equivalence classes of $(4,4)$-pairs}

Combining the results of previous sections yields a
complete solution to the problem of identifying all of the equivalence classes
of $(4,4)$-pairs. It should be observed that any
pattern pair not represented in the table below belongs to a Wilf class of size
one, such classes being determined by numerical evidence (note that there are
$\binom{15}{2}-84=21$ singleton classes in all).

{\footnotesize
\begin{itemize}
\itemsep3pt
\item
$(1121,1232)\simref{\ref{lem-1a21b}}(1112,1223)\simref{\ref{lem-1a21b}}(1121,
1223)\simref{\ref{lem-1a21b}}(1211,1232)\simref{\ref{P45},\ref{lem-1a21b}}\break(1213,
1223)\simref{\ref{P45}}(1213,1234)\simref{\ref{P45}}(1213,1231)\sims(1231,
1234)\simref{\ref{thm-123}}(1232,1234)\simref{\ref{thm-123}}\break(1223,1234)\simref{\ref{thm-123}}
(1233,1234)\simref{\ref{P41}}(1222,1233)\simref{\ref{P41}}(1223,1232)\simref{\ref{P41}}(1223,
1233)\simref{\ref{P41}}\break(1232,1233)\simref{\ref{P46},\ref{P41}}(1231,1232)\simref
{\ref{PrAA1}}(1123,1212)\simref{\ref{PrA2}}(1122,1221)\simref{\ref{PrAA1}}(1212,
1223)\simref{\ref{PrA3}}\break(1122,1223)\simref{\ref{PrA1}}(1221,1223)
\simref{\ref{PrA1}}(1221,1232)\simref{\ref{PrAA1}}(1212,1233)\sims(1221,1233)$

\item \mbox{\cite{MS1}}
$(1212,1232)\sim(1112,1213)\sim(1123,1223)\sim(1221,1231)\sim(1123,
1213)\sim\break(1212,1213)\sim(1212,1221)\sim(1222,1223)\sim(1122,1212)\sim(1222
,1232)\sim(1211,1231)$

\item
$(1112,1233)\sims(1211,1233)\sims(1121,1233)\simref{\ref{lem-1a21b}}(1112,
1234)\sims(1121,1234)\sims\break(1211,1234)\simref{\ref{cor-1222}}(1222,1234)$

\item $(1213,1221)\simref{\ref{P45}}(1213,1232)\simref{\ref{P45}}(1213,1233)\sims(1231,1233)\simref{\ref{PrB}}(1123,1232)\simref{\ref{PrB}}\break(1123, 1233)\simref{\ref{PrB}}(1122,1232)$

\item \mbox{\cite{MS2}} $(1112,1123)\sim(1211,1212)\sim(1122,1123)\sim(1121,1221)\sim(1121,1212)$

\item \mbox{\cite{MS3}} $(1211,1221)\sim(1112,1212)\sim(1212,1222)\sim(1221,1222)$

\item $(1211,1222)\sims(1121,1222)\sims(1112,1222)$

\item $(1111,1121)\sims(1111,1211)\sims(1111,1112)$

\item $(1212,1234)\sims(1221,1234)$

\item $(1112,1232)\simref{\ref{lem-1a21b}}(1211,1223)$

\item $(1222,1231)\sims(1213,1222)$

\item $(1111,1223)\sims(1111,1232)$

\item $(1123,1211)\simref{\ref{tm1},\ref{tm2}}(1123,1222)$

\item $(1123,1234)\simref{\ref{PrC}}(1122, 1233)$

\item $(1121,1231)\simref{\ref{P45},\ref{P46}}(1121,1213)$

\item $(1111,1213)\sims(1111,1231)$

\item $(1111,1212)\sims(1111,1221)$

\item $(1121,1211)\simref{\ref{P43}}(1112,1121)$.
\end{itemize}
}

\section*{Acknowledgement}
We are thankful to Richard Mathar for pointing out an error in an earlier version of this paper.

\end{document}